\documentclass[preprint,10pt]{imsart}
\usepackage[hyperindex,breaklinks]{hyperref}
\usepackage{amsthm,amsmath,natbib}

\RequirePackage[OT1]{fontenc}
\RequirePackage{hyperref}
\usepackage{bbm,bm}
\usepackage{amsmath,amssymb, amsthm}
\usepackage{amsfonts}
\usepackage{graphicx}
\usepackage{subfigure}

\numberwithin{equation}{section}
\theoremstyle{plain}
\newtheorem{thm}{Theorem}[section]
\newtheorem*{thm*}{Theorem}

\newtheorem{lemma}[thm]{Lemma}
\newtheorem{proposition}[thm]{Proposition}

\theoremstyle{definition}

\newtheorem{definition}[thm]{Definition}
\newtheorem*{definition*}{Definition}
\newtheorem{assumption}{Assumption}
\newtheorem{remark}[thm]{Remark}
\newtheorem*{remark*}{Remark}

\newcommand{\dy}{\mathrm{d}y}

\newcommand{\F}{\mathcal{F}}

\newcommand{\NN}{\mathcal{N}}
\newcommand{\E}{\mathbb{E}}
\newcommand{\PP}{\mathbb{P}}
\newcommand{\PPP}{\mathcal{P}}

\newcommand{\II}{\mathbbm{1}}

\newcommand{\field}[1]{\mathbbm{#1}}
\newcommand{\R}{\field{R}}
\newcommand{\C}{\field{C}}

\newcommand{\N}{\field{N}}
\newcommand{\ZZ}{\field{Z}}

\newcommand{\Var}{\mathrm{Var}}

\newcommand{\tmix}{t_{\mathrm{mix}}}

\newcommand{\gammaps}{\gamma_{\mathrm{ps}}}

\newcommand{\mtx}[1]{\bm{#1}}
\newcommand{\dtv}{d_{\mathrm{TV}}}
\newcommand{\diam}{\mathrm{diam\,}}
\newcommand{\econst}{\mathrm{e}}
\newcommand{\hatsigmamax}{\hat{\sigma}_{\mathrm{max}}}
\newcommand{\Lip}{\mathrm{Lip}}

\begin{document}

\begin{frontmatter}

\title{Mixing and Concentration by Ricci Curvature}
\runtitle{Mixing and Concentration by Ricci Curvature}

\begin{aug}
\author{\fnms{Daniel} \snm{Paulin} \ead[label=e1]{paulindani@gmail.com}} 
\runauthor{D. Paulin}

\affiliation{National University of Singapore}

\address{
Department of Statistics \& Applied Probability, \\
National University of Singapore, 6 Science Drive 2, Singapore, 117546, SG.,\\
\printead{e1}}
\end{aug}

\begin{keyword}[class=AMS]
\kwd{60E15}
\kwd{60J05}
\kwd{60J10}
\kwd{68Q87}
\end{keyword}

\begin{keyword}
\kwd{Ricci curvature}
\kwd{Markov chain}
\kwd{concentration inequality}
\kwd{spectral gap}
\kwd{mixing time}
\end{keyword}

\begin{abstract}
We generalise the coarse Ricci curvature method of Ollivier by considering the coarse Ricci curvature of multiple steps in the Markov chain. This implies new spectral bounds and concentration inequalities. We also extend this approach to the bounds for MCMC empirical averages obtained by Joulin and Ollivier. We prove a recursive lower bound on the coarse Ricci curvature of multiple steps in the Markov chain, making our method broadly applicable. Applications include the split-merge random walk on partitions, Glauber dynamics with random scan and systemic scan for statistical physical spin models, and random walk on a binary cube with a forbidden region.
\end{abstract}
\end{frontmatter}

\maketitle

\section{Introduction}
The coarse Ricci curvature of a Markov chain with metric state space $(\Omega,d)$, and kernel $P(x,\mathrm{d}z)$ was defined in \cite{Ollivier2} as
\[\kappa(x,y)=1-\frac{W_{1}(P_x,P_y)}{d(x,y)} \text{ for }x\ne y, \text{ and } \kappa=\inf_{x,y\in \Omega, x\ne y}\kappa(x,y).\]
where $P_x$ denotes the measure $P(x,\mathrm{d}z)$, and $W_1$ denotes the Wasserstein distance of $P_x$ and $P_y$. 

It is known that for reversible chains, $\kappa$ gives a lower bound on the spectral gap: $\gamma\ge \kappa$. It can be also used to bound the mixing time of the chain (known as the Bubley-Dyer path coupling method, see \cite{bubley1997pathcoupling}). The name curvature comes from the fact that it is linked to the geometric definition of Ricci curvature. One of the motivating examples of  \cite{Ollivier2} is the well known Gromov-L\'evy theorem, which it recovers (up to a small constant factor).

When considering Lipschitz functions on $\Omega$ under the stationary distribution $\pi$ of the chain, it is possible to prove variance and concentration bounds, with constants depending on $1/\kappa$,  the typical step size of the Markov chain, and the Lipschitz coefficient. In addition to this, one can show concentration inequalities for MCMC empirical averages of Lipschitz functions (see \cite{Ollivier3}). The coarse Ricci curvature approach has been found to give the right order of concentration and spectral bounds in numerous examples. However, there were also cases where it has not succeeded to give bounds of the correct order. One of them is the split-merge walk on partitions (also called the coagulation-fragmentation chain, see \cite{Diaconissplitmerge} for references), where $\kappa=\mathcal{O}(1/N^2)$, which is too small, since $\gamma=\mathcal{O}(1/N)$ in this case.
In order to extend the coarse Ricci curvature approach to this situation, we define the multi-step coarse Ricci curvature as
\[\kappa_k(x,y)=1-\frac{W_{1}(P_x^k,P_y^k)}{d(x,y)} \text{ for }x\ne y, \text{ and } \kappa_k=\inf_{x,y\in \Omega, x\ne y}\kappa_k(x,y),\]
which is the coarse Ricci curvature of the $k$ step Markov kernel $P^k$.
We extend the spectral and concentration bounds to this case. 
We show that for reversible chains, for any $k\in \N$, the spectral gap satisfies $\gamma\ge \kappa_k/k$, and concentration inequalities hold with constants depending on $\sum_{k=0}^{\infty}(1-\kappa_k)$. In particular, this allows us to recover bounds of the correct order of magnitude for the split-merge walk on partitions.

We propose several approaches to bound $\kappa_k$. The first approach is applicable when the mixing time of the chain can be bounded, and the state space is discrete. In this case, we are  able to obtain bounds on $\kappa_k$ for sufficiently large $k$, which in turn can imply concentration bounds. We illustrate this with an example about the Curie-Weiss model in critical phase.
The second approach gives a recursive lower bound on $\kappa_k$. If the curvature is positive in most of the state space, and negative in a small part, then in some situations, this recursive bound can show that $\kappa_k$ becomes positive for sufficiently large $k$. An example is given about a random walk on a binary cube with a forbidden region.

Now we explain the organisation of this paper. In Section \ref{SecPreliminaries}, we introduce the main definitions.
Section \ref{SecResults} contains our results, in particular, new spectral bounds, concentration inequalities, and moment bounds involving the multi-step coarse Ricci curvature. We also state propositions for bounding $\kappa_k$.
In Section \ref{SecApplications}, we present three applications, the split-merge walk on partitions,  Glauber dynamics with random and systemic scan for statistical physical models, and random walk on a binary cube with a forbidden region. Finally, Section \ref{SecProofs} contains the proofs of our results.

We end the introduction by a few additional remarks about the related literature.
The coarse Ricci curvature approach originates from semigroup tools, which have been used previously in the literature to prove concentration inequalities for Lipschitz functions of random variables distributed according to the stationary distribution of a Markov process (see \cite{Ledoux}, Section 2.3). These can be used to prove concentration for the Gaussian measure, and more generally,  for log-concave densities. For a recent extension of the coarse Ricci curvature to continuous time Markov processes, see \cite{veysseire2012coarse}, and \cite{veysseire2012courbure}. \cite{Veysseirenoneg} obtains concentration bounds in the case when the coarse Ricci curvature is zero. The coarse Ricci curvature has been used previously, but without geometric interpretation, to bound mixing times, known as the Bubley-Dyer path coupling method.
In this sense, it has been also extended to consider multiple steps in the Markov chain, in \cite{DyerPathcouplingextension}, see also \cite{AllanSlyexponential}. The coarse Ricci curvature approach was adapted to graphs in \cite{bauer2011ollivier} and \cite{bauer2013li}, and to adaptive MCMC in \cite{pillai2013finite}.

There is another popular curvature notion called the Sturm-Lott-Villani curvature (\cite{LottVillani}, \cite{Sturm1}). \cite{ollivier2013visual} gives a visual introduction to various curvature definitions, and compares them on numerous examples. 
In the case of Riemannian manifolds, \cite{Milmanisotransport} studies the relation of isoperimetric, functional and transportation cost inequalities, and \cite{milman2011sharp} generalises the Gromov-L\'evy theorem to compact manifolds with negative curvature. This paper was motivated by some of the problems of the survey \cite{Ollivier1}.
Finally, we note that after we have completed this work, \cite{Luczakconcentration} have been bought to our attention. It considers similar ideas as ours, and obtains concentration and spectral bounds depending on the contraction properties of the measures describing multiple steps in the Markov chain. The approach was further developed in \cite{luczak2012quantitative}, \cite{brightwell2013fixed} and \cite{brightwell2013extinction}. Our results in this paper are more precise, since they take into account the typical size of the jump of the Markov chain, as well as the dimension of the state space, which were not considered in the earlier work. In addition, we also show a recursive bound on the multi-step coarse Ricci curvature, which makes our method easier to apply in practice.

\section{Preliminaries}\label{SecPreliminaries}
We will work with stationary, time homogeneous Markov chains $(X_i)_{i\in \N}$ with transition kernel $P(x,\dy)$ taking values in a Polish metric space $(\Omega,d)$. We will denote the stationary distribution of the chain by $\pi$. The expected value of a function $f:\Omega\to \R$ under $\pi$ will be denoted by $\E_{\pi}(f)$. The jump measure when starting from $x$ will be denoted by $P_x$, that is, $P_x(\dy)=P(x,\dy)$. For $k\ge 0$, the $k$-step transition kernel will be denoted by $P^k(x,\dy)$ (in particular, $P^0(x,\dy)=\delta_x(\dy)$, the Dirac-measure concentrated on $x$).

We define the $L^1$ transportation distance (Wasserstein distance) of two measures on $(\Omega,d)$ as
\begin{equation}
W_1(\mu_1,\mu_2):=\inf_{(X,Y)}\E(d(X,Y)),
\end{equation}
where the infimum is taken over all  couplings $(X,Y)$ of $\mu_1$ and $\mu_2$ (i.e. $(X,Y)$ is a random vector taking values on $\Omega\times \Omega$, with marginals $\mu_1$, and $\mu_2$).

To avoid some complications, similarly to Definition 1 of \cite{Ollivier2}, we make the following technical assumptions on the Markov kernel $P(x,\mathrm{d}y)$.
\begin{assumption}[Assumptions on the Markov kernel]\label{assumption1}\hspace{0mm}\\
\begin{enumerate}
\item The measure $P(x,\mathrm{d}y)$ depends measurably on $x\in \Omega$. \\
\item For every $x\in \Omega$, $W_1(\delta_x,P_x)<\infty$.
\end{enumerate}
\end{assumption}

\subsection{Ricci curvature}
The following definition is a generalisation of Ollivier's coarse Ricci curvature (Definition 3 of \cite{Ollivier2}).
\begin{definition}[Multi-step coarse Ricci curvature]\label{defmultistepRicci}
Let $(\Omega, d)$ and $P(x,\dy)$ be as above. Then for $k\in \N$, $x,y\in \Omega$, we let
\begin{equation}\kappa_k(x,y):= 1-\frac{W_1(P^k_x,P^k_y)}{d(x,y)} \text{ if }x\ne y, \text{ and }\kappa_k(x,y):=1 \text{ if } x=y,
\end{equation}
and define the \emph{multi-step coarse Ricci curvature} as $\kappa_k:=\inf_{x,y\in \Omega} \kappa_k(x,y)$.
\end{definition}

For $k=1$, this is just the usual definition of coarse Ricci curvature,  that is, $\kappa=\kappa_1$. For $k=0$, we have $\kappa(x,y)=0$ for $x\ne y$. Then by Proposition 25 of \cite{Ollivier2} it follows that the multi-step coarse Ricci curvature satisfies the inequality 
\begin{equation}\label{kappakleq}1-\kappa_{k+l}\le (1-\kappa_k)(1-\kappa_l) \quad \text{ for }k,l\in \N.\end{equation}

In the case when $\kappa>0$, inequality \eqref{kappakleq} implies that $1-\kappa_i\le  (1-\kappa)^i$. More generally, if we assume that $\kappa>-\infty$, and $\kappa_k>0$ for some $k\in \N$, then \eqref{kappakleq} implies that for any $i\in \N$,
\begin{equation}\label{kappakleq2}1-\kappa_{i} \le (1-\kappa_k)^{\lfloor i/k\rfloor}\cdot M,\end{equation}
for $M:=\sup_{i\in \N}(1-\kappa_i)=\sup_{0\le i<k}(1-\kappa_i)\le (1-\kappa)^{k-1}$, meaning that under the assumption that $\kappa_k>0$, $1-\kappa_i$ decreases at least geometrically fast in $i$.

We note that the condition $\kappa_k>0$ for some $k\in \N$ is weaker than the condition $\kappa>0$ required in the results of \cite{Ollivier2}. This weaker condition also allows some negative curvature (that is, $\kappa<0$). As expected, there is a price to be paid for this weaker assumption. In particular, when compared to the results of \cite{Ollivier2}, the dependence of the constants in our results on $\kappa$ will usually change from $1/\kappa$ to $M k /\kappa_k$. Nevertheless, this weaker assumption enlarges the scope of applicability of the coarse Ricci curvature approach, and as we shall see in the examples, it often allows us to recover results of the correct order of magnitude.

The following proposition shows that the condition $\kappa_k>0$ and some additional weak assumptions guarantee the existence and uniqueness of a stationary distribution $\pi$ (this generalises Corollary 21 of \cite{Ollivier2}).
\begin{proposition}[Existence and uniqueness of stationary distribution]\label{propstatdist}
Let $P(x,\mathrm{d}y)$ be a Markov kernel on a Polish metric space $(\Omega, d)$. Assume that
\begin{enumerate}
\item $\kappa>-\infty$, 
\item $\kappa_k>0$ for some $k\in \N$,
\end{enumerate}
Then the Markov kernel has a unique invariant distribution.
\end{proposition}
\begin{proof}
Let $\PPP(\Omega)$ be the space of all probability measures $\mu$ on $\Omega$ with finite first moments (that is, for some $o\in \Omega$, $\int d(o,x) \mathrm{d}\mu(x)<\infty$). 
On this space, the transportation distance $W_1$ between any two measures is finite, so it is a proper distance, and one can show that $\PPP(\Omega)$ is a Polish space.
Using Proposition 20 of \cite{Ollivier2}, we have that for any two distributions $\mu$, $\mu'$ in $\PPP(\Omega)$, any $i\in \NN$,
\[W_1\left(\mtx{P}^i \mu, \mtx{P}^i\mu'\right)\le (1-\kappa_i)W_1(\mu,\mu').\]
Using \eqref{kappakleq2}, this implies, in particular, that for any $k\in \ZZ_{+}$, $i\in \N$,
\begin{equation}\label{eqW1Pmumup}W_1\left(\mtx{P}^i \mu, \mtx{P}^i\mu'\right)\le (1-\kappa_k)^{\left\lfloor\frac{i}{k}\right \rfloor}\cdot W_1(\mu,\mu').\end{equation}
Using our assumptions and Assumption \ref{assumption1}, this implies that $\mtx{P}^i \mu$ is a Cauchy sequence, and thus using the completeness of $\PPP(\Omega)$, it has a limit, which we call $\pi$. \eqref{eqW1Pmumup} also implies that this limit is the same for any $\mu\in \PPP(\Omega)$. Finally, the stationary follows from the fact that for any $\mu\in \PPP(\Omega)$, \[W_1(\pi, \mtx{P}\pi)\le W_1(\mtx{P}^k \mu, \pi)+W_1(\mtx{P}^k\mu, \mtx{P}\pi)\le W_1(\mtx{P}^k \mu, \pi)+(1-\kappa)W_1(\mtx{P}^{k-1}\mu, \pi), \]
and the right hand side tends to 0 as $k\to \infty$.
\end{proof}

\subsection{Mixing time and spectral gap}
We define the total variational distance of two measures $P, Q$ defined on the same state space $(\Omega,\F)$ as
\begin{equation}\label{dtvdef1}\dtv(P,Q):=\sup_{A\in \F} |P(A)-Q(A)|,\end{equation}\label{dtvdef}
which is equivalent to 
\begin{equation}\label{dtvdef2}\dtv(P,Q):=\inf_{(X,Y)} \PP(X\ne Y),\end{equation}
with the infimum taken over all the couplings $(X,Y)$ of $P$ and $Q$.

We define the mixing time of a time homogeneous Markov chain with general state space in the following way (similarly to Section 4.5 and 4.6 of \cite{peresbook}).
\begin{definition}[Mixing time]\label{mixhom}
Let $X_1,X_2,X_3,\ldots$ be a time homogeneous Markov chain with transition kernel $P(x,\dy)$, state space $\Omega$ (a Polish space), and stationary distribution $\pi$. Let us denote
\[d(t):=\sup_{x\in \Omega} \dtv\left(P^t_x,\pi \right),\hspace{2mm}
\tmix(\epsilon):=\min\{t: d(t)\le \epsilon\},\text{ and }
\hspace{2mm}\tmix:=\tmix(1/4).\]
\end{definition}
Let $L_2(\pi)$ denote the Hilbert space of complex valued measurable functions with domain $\Omega$ that are square integrable with respect to $\pi$, endowed with the inner product $<f,g>_{\pi}=\int f g^* \mathrm{d} \pi$, and norm $\|f\|_{2,\pi}:=\left<f,f\right>_{\pi}^{1/2}=
\E_{\pi}\left(f^2\right)^{1/2}$ (we use the same notation for the induced operator norm). $P$ can be then viewed as a linear operator on $L_2(\pi)$, denoted by $\mtx{P}$, defined as \[(\mtx{P} f)(x):=\E_{P_x}(f),\] and reversibility is equivalent to the self-adjointness of $\mtx{P}$.
The operator $\mtx{P}$ acts on measures to the left, creating a measure $\mu\mtx{P}$, that is, for every measurable subset $A$ of $\Omega$, $\mu \mtx{P}(A):=\int_{x\in \Omega} P(x,A) \mu(\mathrm{d} x)$. For a Markov chain with transition kernel $P(x,dy)$, and stationary distibution $\pi$, we define the \emph{time reversal} of $P$ as the Markov kernel
\begin{equation}\label{Pstardef}
P^*(x, dy):=\frac{P(y,dx)}{\pi(dx)} \cdot \pi(dy).
\end{equation}
Then the linear operator $\mtx{P}^*$ is the adjoint of the linear operator $\mtx{P}$ on $L_2(\pi)$.
For a Markov chain with stationary distribution $\pi$, we define the \emph{spectrum} of the chain as
\[S_2:=\{\lambda\in \C\setminus 0: (\lambda\mathbf{I}-\mtx{P})^{-1}\text{ does not exist as a
bounded lin. oper. on } L_2(\pi)\}.\]
For reversible chains, $S_2$ lies on the real line.
\begin{definition}[Spectral gap and pseudo spectral gap]
The \emph{spectral gap} for reversible chains is
\begin{align*}
\gamma&:=1-\sup\{\lambda: \lambda\in S_2, \lambda\ne 1\}  \quad \text{if eigenvalue 1 has multiplicity 1,}\\
\gamma&:=0 \quad\text{otherwise}.
\end{align*}

For both reversible, and non-reversible chains, the \emph{absolute spectral gap} is
\begin{align*}
\gamma^*&:=1-\sup\{|\lambda|: \lambda\in S_2, \lambda\ne 1\}  \quad \text{if eigenvalue 1 has multiplicity 1,}\\
\gamma^*&:=0 \quad \text{otherwise}.
\end{align*}
In the reversible case, $\gamma\ge \gamma^*$. 

The \emph{pseudo spectral gap} of $\mtx{P}$ (introduced in \cite{Martoncoupling}) is 
\begin{equation}\label{gammapsdef}
\gammaps:=\max_{k\ge 1} \left\{\gamma((\mtx{P}^*)^k \mtx{P}^k)/k\right\}, 
\end{equation}
where $\gamma((\mtx{P}^*)^k \mtx{P}^k)$ denotes the spectral gap of the self-adjoint operator $(\mtx{P}^*)^k \mtx{P}^k$.
\end{definition}
\begin{remark}
The pseudo spectral gap is similar to the spectral gap in the sense that it allows to obtain variance and concentration bounds on MCMC empirical averages, for example $\Var_{\pi}((f(X_1)+\ldots f(X_N))/N)\le 4\Var_{\pi}(f)/(N\gammaps)$ (see \cite{Martoncoupling}, Section 3). Moreover, it is related to the mixing time, $\gammaps\le 1/(2\tmix)$, and for chains on finite state spaces, $\tmix\le (1+2\log(2) + \log(1/\pi_{\min}))/\gammaps$ (here $\pi_{\min}:=\min_{x\in \Omega}\pi(x)$).
\end{remark}

\section{Results}\label{SecResults}
In this section, we will present our results based on the multi-step coarse Ricci curvature. In Section \ref{SecBoundingmultistep}, we present a recursive lower bound for $\kappa_k$. Section \ref{SecSpectralbounds} states spectral bounds, explains the relation of the multi-step coarse Ricci curvature and spectral properties of the Markov chain, while Section \ref{Secdiameterbounds} states bounds on the diameter of the state space. In Section \ref{Secconcentrationbounds}, we state variance, moment, and concentration bounds for Lipschitz functions of random variables distributed according to the stationary distribution of a Markov chain. Finally, in Section \ref{SecMCMC}, we state some error bounds for MCMC empirical averages.
\subsection{Bounding the multi-step coarse Ricci curvature}\label{SecBoundingmultistep}
Our first proposition, the so called geodesic property is useful to get bounds on $\kappa_k$ (similarly as in Proposition 19 of \cite{Ollivier2}).
\begin{proposition}\label{geodesicprop}
Suppose that $(\Omega,d)$ is $\epsilon$-geodesic in the sense that for any two points $x,y\in \Omega$, there exists an integer $n$, and a sequence $x_0=x$,$x_1,\ldots,x_n=y$ such that $d(x,y)=\sum_{i=0}^{n-1} d(x_i,x_{i+1})$ and $d(x_i,x_{i+1})\le \epsilon$ for $0\le i\le n-1$.
Let $k\ge 1$, then if $\kappa_k(x,y)\ge \kappa_k$ for any pair of points $x,y$ with $d(x,y)\le \epsilon$, then $\kappa_k(x,y)\ge \kappa_k$ for any pair of points $x,y\in \Omega$.
\end{proposition}
\begin{proof}
Apply Proposition 19 of \cite{Ollivier2} to the Markov kernel $P^k$.
\end{proof}

The following proposition gives a recursive lower bound on the multi-step Ricci curvature $\kappa_k(x,y)$.
\begin{proposition}\label{kstepricciestimateprop}
For some $x,y\in \Omega, x\ne y$, let $(X,Y)$ be a coupling of $P_x$ and $P_y$, then
\[\kappa_{k+1}(x,y)\ge 1-\E\left(\frac{d(X,Y) (1-\kappa_k(X,Y))}{d(x,y)}\right).\]
If $(X,Y)$ satisfies that $\E(d(X,Y))=W_1(P_x,P_y)$ (that is, the coupling ``achieves" the Wasserstein distance), then 
\[\kappa_{k+1}(x,y)\ge \kappa(x,y)+\E\left(\frac{\kappa_{k}(X,Y) d(X,Y)}{d(x,y)}\right).\]
\end{proposition}
\begin{proof}
We are going to construct a coupling $X_{k+1}\sim P_x^{k+1}, Y_{k+1}\sim P_y^{k+1}$ as follows. 
We start from our coupling $(X,Y)$ of $P_x$ and $P_y$, 
and for any $a,b\in \Omega$, define \[\mathcal{L}(X_{k+1},Y_{k+1}|X=a,Y=b)\]
as the optimal coupling between $P^k_a$, $P^k_b$, achieving $\E(d(X_{k+1},Y_{k+1})|X=a,Y=b)=W_1(P^k_a, P^k_b)$. Then we have
\begin{align*}W_1(P^{k+1}_x, P^{k+1}_y)&\le \E(d(X_{k+1},Y_{k+1}))=\E(\E(d(X_{k+1},Y_{k+1})|X,Y))\\
&= \E((1-\kappa_k(X,Y))d(X,Y)),\end{align*}
and thus
\begin{align*}
\kappa_{k+1}(x,y)&=1-\frac{W_1(P^{k+1}_x,P^{k+1}_y)}{d(x,y)}\ge 1-\frac{\E((1-\kappa_k(X,Y))d(X,Y))}{d(x,y)}\\
&=1-\frac{\E(d(X,Y))}{d(x,y)}+\E\left(\frac{\kappa_{k}(X,Y) d(X,Y)}{d(x,y)}\right).
\end{align*}
Finally, if $(X,Y)$ is the optimal coupling between $P_x$, and $P_y$, then
$\E(d(X,Y))=(1-\kappa(x,y))d(x,y)$, and the second claim of the proposition follows.
\end{proof}
Suppose that everywhere except in a small part of the state space $\Omega$, $\kappa(x,y)>0$ for neighbouring $x$ and $y$. Then this result says that $\kappa_{k+1}(x,y)$ can be lower bounded by some sort of average of $\kappa_{k}(x,y)$, and for sufficiently large $k$, the negative curvature may disappear. In Section \ref{SecBinaryCube}, we are going to apply this result to a random walk on the binary cube with a forbidden region. 

\subsection{Spectral bounds}\label{SecSpectralbounds}
Our first result is a bound on the mixing time.
\begin{proposition}[Relation of mixing time and coarse Ricci curvature]\label{mixingRicciprop}
Let $(\Omega, d)$ be a metric space, and $P(x,\dy)$ a Markov kernel, as previously. Suppose that $\mathrm{diam\,}(\Omega)<\infty$, and there is $d_0>0$ such that for any $x\ne y$, $d(x,y)\ge d_0$. Then
\begin{equation}\label{tmixkappabound}\tmix(\epsilon)\le \inf \{ k: k\in \N, 1-\kappa_k\le \epsilon d_0/\diam(\Omega) \}.\end{equation}
Conversely, we have, for any $\epsilon>0$, $k\ge \tmix(\epsilon/2)$,
\begin{equation}\label{kappatmixbound}\kappa_k\ge 1-\epsilon \cdot \diam(\Omega)/d_0.
\end{equation}
\end{proposition}
\begin{remark}
If $\kappa>0$, then $1-\kappa_k\le (1-\kappa)^k$, thus
\[\tmix(\epsilon)\le \left\lceil \frac{\log(\epsilon d_0/\mathrm{diam\,}(\Omega))}{\log(1-\kappa)}\right\rceil,\]
which is the well known Bubley-Dyer path coupling bound. Our bound, however, does not require $\kappa>0$, thus it is more general.
\end{remark}
\begin{proof}[Proof of Proposition \ref{mixingRicciprop}]
For two disjoint $x,y \in \Omega$, $k\ge 1$, we have
\[\dtv(P_x^k,P_y^k)\le \frac{W_1(P_x^k,P_y^k)}{d_0}\le \frac{\diam(\Omega)(1-\kappa_k)}{d_0}.\]
Averaging out in $y$ gives
\[\dtv(P_x^k,\pi)\le \sup_{y\in \Omega}\frac{W_1(P_x^k,P_y^k)}{d_0}\le \frac{\diam(\Omega)(1-\kappa_k)}{d_0},\]
and this is less than equal to $\epsilon$ if $1-\kappa_k\le \epsilon d_0/\diam(\Omega)$.
The proof of \eqref{kappatmixbound}, based on Proposition \ref{geodesicprop}, is left to the reader as exercise.
\end{proof}
Now we give lower bounds on the spectral gap and the pseudo spectral gap.
\begin{proposition}[Relation of spectral gap and coarse Ricci curvature]\label{kappagammaprop}
For reversible chains, for every $k\ge 1$,
\begin{equation}\label{gammastareq1}
\gamma^*\ge 1- (1-\kappa_k)^{1/k}\ge \frac{\kappa_k}{k}.
\end{equation}
Without assuming reversibility, for every $k\ge 1$,
\begin{equation}\label{gammapseq1}
\gammaps\ge \frac{1-(1-\kappa_k(P^*))(1-\kappa_k)}{k},
\end{equation}
with $\kappa_k(P^*)$ denoting the $k$th step coarse Ricci curvature of the time reversal of our Markov kernel, $P^*(x,\dy)$.
\end{proposition}
\begin{remark}
In Section \ref{secsplitmergewalk}, we are going to use this result to obtain a lower bound for the spectral gap of the split-merge walk on partitions. Another application is given in Section \ref{CurieWeissbound}, where we use this proposition to bound the pseudo spectral gap of the systemic scan Glauber dynamics in the high temperature regime. 
\end{remark}
\begin{proof}[Proof of Proposition \ref{kappagammaprop}]
For reversible chains, by applying Proposition 30 of \cite{Ollivier2} to $P^k$, we have
$1-\gamma^*(P^k)\le 1-\kappa_k$, and \eqref{gammastareq1} follows by the fact that
$1-\gamma^*(P^k)=(1-\gamma^*)^k$.
Similarly, applying Proposition 30 of \cite{Ollivier2} to the reversible kernel $(P^*)^kP^k$, we have
$1-\gamma^*((P^*)^kP^k)\le 1-\kappa((P^*)^kP^k)$. Using Definition \ref{defmultistepRicci}, one can show that $1-\kappa((P^*)^kP^k)\le (1-\kappa_k(P^*))(1-\kappa_k)$, thus \eqref{gammapseq1} follows.
\end{proof}

\subsection{Diameter bounds}\label{Secdiameterbounds}
Our first result in this section is an analogue of Proposition 23 of \cite{Ollivier2}.
\begin{proposition}[$L^1$ Bonnet-Myers theorem]\label{L1BonnetMyers}
For $k\ge 1$, let the \emph{$k$-step jump length} of the random walk at $x$ be
\[J_k(x):=\int_{y} d(x,y) \mathrm{d}P_{x}^k(y).\]
Suppose that for some $k\ge 1$, $\kappa_k(x,y)>0$ for every $x,y\in \Omega$. Then for every $x,y\in \Omega$, we have
\[d(x,y)\le \frac{J_k(x)+J_k(y)}{\kappa_k(x,y)},\]
and in particular,
\[\diam(\Omega)\le \frac{2\sup_x J_k(x)}{\kappa_k}\le  \frac{2k \sup_x J(x)}{\kappa_k}.\]
\end{proposition}
\begin{proof}
Apply Proposition 23 of \cite{Ollivier2} to $P^k$.
\end{proof}
\begin{remark}
In Section \ref{secsplitmergewalk}, we are going to apply this proposition to split-merge walk on partitions, and obtain a bound on diameter of $\Omega$ of $\mathcal{O}(N)$. 
\end{remark}
Similarly, we can generalise Proposition 24 of \cite{Ollivier2}.
\begin{proposition}[Average $L^1$ Bonnet-Myers theorem]\label{L1AverageBonnetMyers}
Suppose that for some $k\ge 1$, $\kappa_k(x,y)>0$ for all $x,y\in \Omega$. Then for any $x\in \Omega$, we have
\[\int d(x,y)\mathrm{d}\pi(y)\le \frac{J_k(x)}{\kappa_k},\]
and thus,
\[\int\int d(x,y)\mathrm{d}\pi(x)\mathrm{d}\pi(y) \le \frac{2\inf_x J_k(x)}{\kappa_k}.\]
\end{proposition}
\begin{proof}
Apply Proposition 24 of \cite{Ollivier2} to $P^k$.
\end{proof}
\begin{remark}
In Section \ref{SecMCMC}, we are going to use this proposition for our bounds about MCMC empirical averages for non-stationary initial distributions.
\end{remark}

\subsection{Concentration bounds}\label{Secconcentrationbounds}
Similarly to the results of \cite{Ollivier2}, our concentration bounds will be based on 3 types of quantities related to the multi-step coarse Ricci curvature, the average step size of the Markov chain, and the dimension of the state space. In order to avoid unnecessary repetitions in the statement of the theorems, we introduce some notations (similarly to Definition 18 of \cite{Ollivier2}).
\begin{definition}\label{manydefs}
Firstly, we make a few definitions related to the multi-step coarse Ricci curvature.
Let us define, for any $x,y\in \Omega$,
\[\kappa_{\Sigma}^c(x,y):=\sum_{k=0}^{\infty}(1-\kappa_k(x,y)), \text{ let }\kappa_{\Sigma}^c:=\sum_{k=0}^{\infty}(1-\kappa_k), \text{ and } M:=\sup_{k\ge 0}(1-\kappa_k).
\]
The letter $c$ refers to complement (we add up $1-\kappa_k(x,y)$ instead of $\kappa_k(x,y)$).

Secondly, we state some definitions related to the step size of the Markov chain.
Let the \emph{(coarse) diffusion constant} of the random walk at $x$ be
\[\sigma(x):=\left(\frac{1}{2}\int\int d(y,z)^2 \mathrm{d}P_x(y) \mathrm{d}P_x(z)\right)^{1/2},\]
and let the \emph{average diffusion constant} be
\[\sigma:=\left(\E_\pi(\sigma^2)\right)^{1/2}.\]
Similarly, define the \emph{mean square jump length} as
\[\hat{\sigma}(x):=\left(\int_y d(x,y)^2 \mathrm{d}P_x(y)\right)^{1/2},\]
and the \emph{average mean square jump length} as
\[\hat{\sigma}:=\left(\E_{\pi}\left(\hat{\sigma}^2\right)\right)^{1/2}.\]
Let the \emph{local granularity} be $\sigma_{\infty}(x):=\frac{1}{2} \mathrm{diam\, Supp\,} P_x$ (the diameter of the support of $P_x$), and the \emph{granularity} be $\sigma_{\infty}:=\sup_{x\in \Omega}\sigma_{\infty}(x)$. 
Define the \emph{maximal diffusion constant} as $\sigma_{\mathrm{max}}=\sup_{x\in \Omega}\sigma(x)$, and the \emph{maximal mean square jump length} as $\hat{\sigma}_{\mathrm{max}}=\sup_{x\in \Omega}\hat{\sigma}(x)$.

Finally, we let the \emph{local dimension} at $x$ be
\[n(x):=\frac{\sigma(x)^2}{\sup\{\Var_{P_x} f, f: \mathrm{Supp\,} P_x \to \R \text{ 1 - Lipschitz}\}}.\]
\end{definition}

\begin{remark}
$\kappa_{\Sigma}^c$ will take the place of $1/\kappa$ in some of our results. It can be bounded using  \eqref{kappakleq} and \eqref{kappakleq2} as
\begin{equation}\label{eqkappasigmacbound}
\kappa_{\Sigma}^c\le \frac{\sum_{i=0}^{k-1}(1-\kappa_i)}{\kappa_{k}}\le  \frac{kM}{\kappa_{k}} \text{ for any }k\ge 1.
\end{equation}
The random walk can be divided into a drift term (corresponding to the change of the expected location), and a diffusion term (corresponding to the spread in space). The diffusion constant $\sigma^2(x)$ quantifies the diffusion term, when starting from point $x$. The local dimension $n(x)$ is a quantity related to the dimension of the state space $\Omega$.  It is always true that $n(x)\ge 1$.
\end{remark}

Our first concentration result is a variance bound for Lipschitz functions (generalising Proposition 32 of \cite{Ollivier2}).
\begin{thm}[Variance bound]\label{Varianceboundthm}
For any 1-Lipschitz function $f$ on $(\Omega, d)$, we have
\begin{equation}\label{nonrevvareq}\Var_{\pi}(f)\le \left(\sum_{k\ge 0} (1-\kappa_k)^2\right)\E_{\pi}\left(\frac{\sigma^2}{n}\right).\end{equation}
\end{thm}
When compared to Proposition 32 of \cite{Ollivier2}, in our bound, $\frac{1}{\kappa}$ is replaced by $\sum_{k\ge 0} (1-\kappa_k)^2$.

Our next result is a moment bound for Lipschitz functions of reversible chains.
\begin{thm}[Moment bound for reversible chains]\label{Revmomentboundthm}
Suppose that
\begin{equation}\label{revvarcond}\int_{y}d(x,y) \kappa_{\Sigma}^c(x,y)\mathrm{d}P_x(y) <\infty\text{ for } \pi - \text{almost every } x\in \Omega.\end{equation}
Then for any 1-Lipschitz function $f$ on $(\Omega, d)$, for any $p\ge 1$, we have
\[\E_{\pi}\left([f(X)-\E_{\pi}(f)]^{2p}\right)\le 
\left(\frac{(2p-1)\kappa_{\Sigma}^c}{2}\right)^p\cdot \E_{\pi}(\hat{\sigma}^{2p}).\]
\end{thm}
Now we state a concentration bound for reversible chains.
\begin{thm}[Concentration for reversible chains]\label{revconcthm}
For reversible chains satisfying \eqref{revvarcond}, for any 1-Lipschitz function $f$ on $(\Omega, d)$ we have the Gaussian bound
\begin{equation}\label{revconceq1}
\PP_{\pi}(|f(X)-\E_{\pi}(f)|\ge t)\le 2\exp\left(-\frac{t^2}{\kappa_{\Sigma}^c\cdot \hatsigmamax^2}\right).
\end{equation}
Let $V:\Omega\to \R_+$ be a function satisfying that for every $x\in \Omega$,
\begin{equation}\int \kappa_{\Sigma}^c(x,y) d(x,y)^2 \mathrm{d}P_{x}(y) \le V(x).\end{equation}
Let $L:=8\E_{\pi}(V)/(\|V\|_{\Lip}^2\hat{\sigma}_{\max}^2 \kappa_{\Sigma}^c)$, where $\|V\|_{\Lip}$ is the Lipschitz coefficient of $V$. Then for any $t\ge 0$,
\begin{equation}\label{revconceq2}
\PP_{\pi}(|f(X)-\E_{\pi}(f)|\ge t)\le 2\exp\left(-\frac{t^2}{4\E_{\pi}(V)+4L^{-1/2}\cdot t}\right).
\end{equation}
\end{thm}

More generally, without using reversibility, we have the following concentration bound (generalising Theorem 33 of \cite{Ollivier2}).
\begin{thm}[Concentration without reversibility]\label{nonrevconcthm}
For any function $f$ with Lipschitz-coefficient $\|f\|_{\Lip}$ on $(\Omega, d)$,
let $S_{\max}:=\sup_{x\in \Omega}\frac{\sigma^2(x)}{n(x)}$, and
denote
\[D_{\max}:=2\|f\|_{\Lip}^2 S_{\max} \cdot \sum_{i=0}^{\infty}\exp\left(\frac{2}{3} (1-\kappa_i) \|f\|_{\Lip}\right) (1-\kappa_i)^2.\]
Let $t_{\max}:=D_{\max}/(6\sigma_{\infty})$, then for $0\le t\le t_{\max}$, we have the Gaussian bound
\begin{equation}\label{nonrevgaussbound}
\PP_{\pi}(|f(X)-\E_{\pi}(f)|\ge t)\le 2\exp\left(-\frac{t^2}{D_{\max}}\right),
\end{equation}
while for $t> t_{\max}$, we have the exponential bound
\begin{equation}\label{nonrevexpbound}
\PP_{\pi}(|f(X)-\E_{\pi}(f)|\ge t)\le 2\exp\left(-\frac{t_{\max}^2}{D_{\max}}-\frac{t-t_{\max}}{3\sigma_{\infty}}\right).
\end{equation}
\end{thm}
\begin{thm}[Second concentration bound without reversibility]\label{nonrevconcthm2}
Alternatively, suppose that $\sigma^2(x)/n(x)\le S(x)$ for some $S:\Omega\to \R$ (for every $x\in \Omega$). Let $K$ be a positive integer such that $\kappa_K>0$. Let 
\[D:=\|f\|_{\Lip}^2 \E_{\pi}(S) \cdot \frac{16M^2 K}{\kappa_K-\kappa_K^2/4}.\]
Let $\lambda_{\max}':=\min\left(\frac{1}{6M\sigma_{\infty}\|f\|_{\Lip}}, \frac{\kappa_K}{4KM^2\|S\|_{\Lip}\|f\|_{\Lip}}\right)$, and $t_{\max}':=D \lambda_{\max}'/2$. Then for $0\le t\le t_{\max}'$, 
\begin{equation}\label{nonrevgaussbound2}
\PP_{\pi}(|f(X)-\E_{\pi}(f)|\ge t)\le 2\exp\left(-\frac{t^2}{D}\right),
\end{equation}
while for $t> t_{\max}'$,
\begin{equation}\label{nonrevexpbound2}
\PP_{\pi}(|f(X)-\E_{\pi}(f)|\ge t)\le 2\exp\left(-\frac{t_{\max}^2}{D}-(t-t_{\max}')\cdot \lambda_{\max}'\right).
\end{equation}
\end{thm}
\begin{remark}
There are some important differences between the inequalities with or without reversibility. Firstly, in the reversible case, Theorem \ref{revconcthm} is not using the maximal jump diameter $\sigma_{\infty}$, thus it may give better bounds than Theorem \ref{nonrevconcthm} in cases when $\sigma_{\infty}$ is very large (or infinity) compared to the typical jump length. Moreover, we have also obtained moment bounds under similar conditions in Theorem \ref{Revmomentboundthm}. However, these results ignore the local dimension $n(x)$, while 
Theorems \ref{Varianceboundthm} and \ref{nonrevconcthm} take it into account, and thus they can give better bounds when $n(x)\gg 1$. The variance, moment, and concentration bounds above can be applied to most of our examples in Section \ref{SecApplications}.
\end{remark}

\subsection{MCMC empirical averages}\label{SecMCMC}
\cite{Ollivier3} has proven variance bounds and concentration inequalities for MCMC empirical averages under the positive Ricci curvature assumption ($\kappa>0$). Here we generalise these results using the multi-step Ricci curvature, and then compare them with inequalities in the literature obtained by spectral methods (for the proofs, see Section \ref{secproofmcmc}).

Our results will prove variance and concentration bounds for the empirical averages $Z:=\left(\sum_{i=t_0+1}^{N}f(X_i)\right)/(N-t_0)$, here $t_0$ is the so called ``burn-in" time.
For some initial distribution $X_1\sim q$, our results will show concentration of $Z$ around its mean denoted by $\E_{q}(Z)$.
The following proposition (generalising Proposition 1 of \cite{Ollivier3}) bounds the bias $|\E_{q}(Z)-\E_{\pi}(f)|$.
\begin{proposition}[Bound on the bias of MCMC empirical averages]\label{biasboundprop}
For any Lipschitz $f:\Omega\to \R$, for initial distribution $X_1\sim q$, we have
\[|\E_{q}(Z)-\E_{\pi}(f)|\le \frac{(1-\kappa_{t_0+1})\kappa_{\Sigma}^c}{N-t_0}W_{1}(q,\pi)\|f\|_{\Lip}.\]
\end{proposition}

The bound on the bias depends on the term $W_1(q,\pi)$, the Wasserstein distance of $q$ and $\pi$. Let $\delta_x$ be the measure concentrated on the point $x$. Then the Wasserstein distance of $\delta_x$ and $\pi$ is called the \emph{eccentricity} of $x$ (see \cite{Ollivier2}, \cite{Ollivier3}), and denoted by $E(x):=\int_{y\in \Omega}d(x,y)\mathrm{d}\pi(y)$. Now by definition, we have
\[W_1(q,\pi)\le \int_{y\in \Omega} E(y)\mathrm{d}q(y).\]
Moreover, it follows from \cite{Ollivier2}, and Proposition \ref{L1AverageBonnetMyers}, that
\begin{equation}
\label{ExboundL1BonnetMyers}\E(x)\le \left\{
\begin{array}{l}\mathrm{diam} (\Omega) \\E(x_0)+d(x,x_0) \text{ for every }x_0\in \Omega\\ 
J_k(x)/\kappa_k\text{ for any }k \ge 1.\end{array}\right.
\end{equation}
Our next result is a variance bound (generalising Theorems 2 and 3 of \cite{Ollivier3}).
\begin{thm}[Variance bound for MCMC empirical averages]\label{MCMCVarRiccithm}
Let 
\begin{equation}\label{Zdefeq}Z:=\left(\sum_{i=t_0+1}^{N}f(X_i)\right)/(N-t_0),\end{equation}
and let $K$ be a positive integer such that $\kappa_K>0$. Suppose that $\sigma^2(x)/n(x)\le S(x)$ for some $S:\Omega\to \R$ for every $x\in \Omega$, then
\begin{align*}
&
\Var_{q}(Z)\le \|f\|_{\Lip}^2\cdot \frac{\sum_{i=0}^{K-1}(1-\kappa_i)^2}{N-t_0}\cdot \frac{K}{\kappa_K^2}\cdot \E_{\pi}(S) \left(1+ \II[t_0>0]\cdot \frac{KM^2}{\kappa_K(N-t_0)}\right)\\
&\quad\nonumber+\|f\|_{\Lip}^2\|S\|_{\Lip}W_{1}(q,\pi)M^3 \cdot \frac{2 K^2}{\kappa_K}\cdot (1-\kappa_K)^{t_0/K - 3}.
\end{align*}
\end{thm}
\begin{remark}
In the stationary case, we can set $t_0=0$, so the bound becomes
\[\Var_{\pi}(Z)\le \|f\|_{\Lip}^2\cdot \frac{\sum_{i=0}^{K-1}(1-\kappa_i)^2}{N}\cdot \frac{K}{\kappa_K^2}\cdot \E_{\pi}\left(\frac{\sigma^2}{n}\right).\]
\end{remark}

The following results prove concentration inequalities for the empirical averages (generalising Theorems 4 and 5 of \cite{Ollivier3}).
\begin{thm}[Concentration bound for MCMC empirical averages]\label{MCMCconcRiccithm}
Let $X_1,X_2,\ldots$ be a Markov chain with initial distribution $q$ ($X_1\sim q$).
Let $Z$ be as in \eqref{Zdefeq}, and let $K$ be a positive integer such that $\kappa_K>0$. Let $S_{\max}:=\sup_{x\in \Omega}\frac{\sigma^2(x)}{n(x)}$.
Let $\lambda_{\max}:=\frac{\kappa_K(N-t_0)}{12M K\sigma_{\infty}\|f\|_{\Lip}}$,
\begin{align*}
t_{\max}&:= \econst^{1/6} \|f\|_{\Lip} \frac{(N-t_0+K-1)+\II[t_0>0] K/\kappa_K}{(N-t_0)}\cdot \frac{\sum_{i=0}^{K-1}(1-\kappa_i)^2}{12M \sigma_{\infty}\kappa_K},\text{ and }\\
D_{\max}&:=2\econst^{1/6} \|f\|_{\Lip}^2 S_{\max} \frac{(N-t_0+K-1)+\II[t_0>0] K/\kappa_K}{(N-t_0)^2}\cdot \frac{K}{\kappa_K^2}\sum_{i=0}^{K-1}(1-\kappa_i)^2.
\end{align*}
Then for $0\le t\le t_{\max}$, we have the Gaussian bound
\begin{equation}\label{nonrevgaussbound}
\PP_{q}(|Z-\E_{\pi}(f)|\ge t)\le 2\exp\left(-\frac{t^2}{D_{\max}}\right),
\end{equation}
while for $t> t_{\max}$, we have the exponential bound
\begin{equation}\label{nonrevexpbound}
\PP_{q}(|Z-\E_{\pi}(f)|\ge t)\le 2\exp\left(-\frac{t_{\max}^2}{D_{\max}}-\lambda_{\max}(t-t_{\max})\right).
\end{equation}
\end{thm}
\begin{thm}[Second concentration bound for MCMC empirical averages]\label{MCMCconcRiccithm2}
Alternatively, suppose that $\sigma^2(x)/n(x)\le S(x)$ for some function $S:\Omega\to \R$ (for every $x\in \Omega$). Denote 
\begin{align*}
\lambda_{\max}'&:=\frac{1}{4KM}\cdot \frac{\kappa_K(N-t_0)}{\|f\|_{\Lip}}\cdot \min(1/(3\sigma_{\infty}), \kappa_K/(4KM^2\|S\|_{\Lip})), \\
t_{\max}'&:=\min(1/(3\sigma_{\infty}), \kappa_K/(4KM^2\|S\|_{\Lip})) \cdot \frac{8M\|f\|_{\Lip} K\E_{\pi}(S)}{\kappa_K},\text{ and}\\
D&:=\frac{64M^2\|f\|_{\Lip}^2 K^2}{\kappa_K^2(N-t_0)} 
\cdot \Bigg[\E_{\pi}(S)\left(1+\II[t_0>0]\cdot \frac{2K}{(N-t_0)\kappa_K}\right)\\
&+M\|S\|_{\Lip} W_1(q,\pi) \cdot \frac{3K}{(N-t_0)\kappa_K}\cdot\left(1-\kappa_K/2\right)^{\lfloor t_0/K\rfloor-3}\Bigg].
\end{align*}
Then for $0\le t\le  t_{\max}'$, we have the Gaussian bound
\begin{equation}\label{nonrevgaussbound}
\PP_{q}(|Z-\E_{\pi}(f)|\ge t)\le 2\exp\left(-\frac{t^2}{D}\right),
\end{equation}
while for $t> t_{\max}'$, we have the exponential bound
\begin{equation}\label{nonrevexpbound}
\PP_{q}(|Z-\E_{\pi}(f)|\ge t)\le 2\exp\left(-\frac{t_{\max}^2}{D}-\lambda_{\max}'(t-t_{\max})\right).
\end{equation}
\end{thm}
\begin{remark}
For stationary chains, we can set $t_0=0$, and $W_1(q,\pi)=0$, so the terms including them can be omitted.
\end{remark}
Now we briefly compare these results to some of the variance bounds and Bernstein inequalities stated in the literature. \cite{RudolfLazy}, \cite{rudolf2012explicit}, and \cite{novak2014computation} obtain explicit bounds on variance and the burn-in time $t_0$ under various moment conditions on the function $f$ and the density ratio $\frac{\mathrm{d}\mu}{\mathrm{d}\pi}$ between the initial and the stationary distributions. These bounds are applicable for uniformly ergodic, as well as geometrically ergodic reversible chains, with explicit examples are given for random walks on convex subsets of $\R^d$, and log-concave densities. 
The coarse Ricci curvature approach has stronger assumptions ($\kappa_k>0$ and $f$ has to be Lipschitz), but the bounds essentially only depend on the parameters $\kappa_k$ and $\|f\|_{\Lip}$, whereas in Rudolf's approach one also needs to bound or estimate various moments of $f$.
Moreover, the coarse Ricci curvature approach can also show exponential concentration bounds.

Variance and concentration bounds were also shown in Section 3 of \cite{Martoncoupling}, based on the spectral gap. If the function $f$ is Lipschitz, and its range is quite large, then the range of the Gaussian tails ($t_{\max}$) can be much larger than what we would get from the Bernstein-type inequalities of \cite{Martoncoupling}. Moreover, our results only use the Lipschitz coefficient of $f$, and the coarse Ricci curvature, which can be bounded theoretically. On the other hand, the Bernstein-type bounds also require the variance, and asymptotic variance of $f$, which are usually difficult to compute theoretically, and need to be estimated numerically (see Section 3 of \cite{nonasymptotic} for more details). The main disadvantage of the Ricci curvature approach is that it is very sensitive to the Lipschitz coefficient of $f$, and may be less precise for non-smooth functions than Bernstein-type inequalities.

\section{Applications}\label{SecApplications}
In this section, we present some applications of our results.
Firstly, in Section \ref{secsplitmergewalk}, we use the multi-step Ricci curvature (in particular, Proposition \ref{kappagammaprop} and Theorem \ref{revconcthm}) to prove spectral bounds for the transposition walk on the symmetric group, and get concentration inequalities for Lipschitz functions of uniform permutations.
In Section \ref{SecStatfiz} we apply our theorems to Markov chains related to statistical
physical models. First, in Section \ref{SectionBoundsusingDobrushin}, we show how Dobrushin's interdependence matrix is related to the multi-step Ricci curvature, for Glauber dynamics with random scan and systemic scan. In Sections \ref{CurieWeissbound} and \ref{Sec1DIsing}, we apply these bounds to the Curie-Weiss and 1D Ising models, respectively.
Finally, in Section \ref{SecBinaryCube}, we present an application of the recursive lower bound for $\kappa_k$ to a random walk on a binary cube with a forbidden region.

\subsection{Split-merge random walk on partitions}\label{secsplitmergewalk}
The partitions of $N$ are $m$-tuples of positive integers $(a_1,\ldots,a_m)$, such that $a_1\ge a_2\ge \ldots \ge a_m$, $\sum_{i=1}^{m}a_i=N$, and $m\le N$. Let us denote the set of the partitions of $N$ by $\Omega$. The split-merge random walk can be thought as the projection of the transposition random walk on the symmetric group $S^N$ to the partitions of $N$, according to the cycle structure of the permutations. 
The split-merge walk is defined as in Definition 2 of \cite{bormashenko2011coupling}, as follows.

Assume that we are in $(a_1,\ldots,a_m)$. Then in the following step, we may
\begin{enumerate}
\item \emph{Split} -- $a_i$ is replaced by $(r,a_i-r)$, with probability $a_i/n^2$ for every $1\le i\le m, 1\le r\le a_i-1$.
\item \emph{Merge} -- Replace $a_i$ and $a_j$ with $a_i+a_j$, with probability $2a_i a_j/N^2$, for every $1\le i<j\le m$.
\item \emph{Stay} -- stay in place with probability $1/N$.
\end{enumerate}
The stationary distribution $\pi$ of this random walk can be written as 
\[\pi((a_1,\ldots,a_m))=\frac{\text{\# of permutations in }S^N\text{ with cycle structure }(a_1,\ldots,a_m)}{N!}.\]
For $x,y\in \Omega$, we define the distance $d(x,y)$ as the minimal number of splits or merges required to get from $x$ to $y$ (or vice-versa). The following proposition estimates the multi-step Ricci curvature $\kappa_k$ for this random walk on the metric space $(\Omega, d)$. 
\begin{proposition}[Ricci curvature for the split-merge walk on partitions]
For the split-merge walk on partitions of $N$, $\kappa>0$, and thus $\kappa_i> 0$ for any $i\ge 1$. Moreover, there exists $\alpha>0$, $0<\beta<1$ universal constants such that for $k\ge (\alpha+1/2) N$, $\kappa_k\ge \beta$.
\end{proposition}
\begin{proof}
First, we are going to show that $\kappa>0$. By Proposition \ref{geodesicprop}, it is sufficient to show that
\begin{equation}\label{permutationkappaeq}
W_1(P_x,P_y)\le (1-\kappa)d(x,y),
\end{equation}
for neighbouring $x$ and $y$, that is, when $d(x,y)=1$. Now it is easy to construct a coupling $(X,Y)$ of  $P_x$ and  $P_y$ such that $d(X,Y)\le 1$, and $\PP(X=Y)=2/N^2$. This means that \eqref{permutationkappaeq} holds with $\kappa=2/N^2$.
The fact that $\kappa_k\ge \beta$ for $k\ge (\alpha +1/2)N$ follows from Lemma 17 of \cite{bormashenko2011coupling}.
\end{proof}

Now we can apply our results on this example. Firstly, using Proposition \ref{mixingRicciprop}, and the facts that $\diam(\Omega)=N-1$, $d_0=1$, and $1-\kappa_{(\alpha +1/2)N\cdot l}\le (1-\beta)^l$ for $l\in \N$, we have
\[\tmix(\epsilon)\le (\alpha +1/2)N \cdot \frac{\log((N-1)/\epsilon)}{\log(1/(1-\beta))}=\mathcal{O}(N\log(N)).\]
Similarly, using Proposition \ref{kappagammaprop}, we can see that $\gamma^*\ge \frac{\kappa_k}{k}\ge \frac{\beta}{(\alpha +1/2)N}=\mathcal{O}(1/N)$.
These are likely to be of the correct order of magnitude, since similar results hold for the transposition walk on the symmetric group (as shown in \cite{diaconisgeneratingperm}). Such bounds could not have been deduced using original coarse Ricci curvature approach of \cite{Ollivier2}, since $\kappa=\mathcal{O}(1/N^2)$.

Applying Proposition \ref{L1BonnetMyers} shows that $\diam(\Omega)\le 2 (\alpha+1/2)N/\beta$, which is the correct order of magnitude.

Finally, in our concentration bounds for reversible chains (Theorem \ref{Varianceboundthm} and Theorem \ref{revconcthm}), we have $\kappa_{\Sigma}^c\le (\alpha +1/2)N/\beta$, thus for any $f:\Omega\to \R$ that is 1-Lipschitz with respect to $d$,
$\Var_{\pi}(f)\le (\alpha +1/2)N/\beta$ and
\[\PP_{\pi}(|f(X)-\E_{\pi}f|\ge t)\le \exp\left(-\frac{t^2}{(\alpha +1/2)N/\beta}\right).\]
Note that this result also follows from the concentration result for functions of random permutations (see \cite{Maurey}, and \cite{Talagrand1}), since the $d(x,y)$ can be bounded from above by the transposition distance.

It would be interesting to prove similar bounds for the transposition walk on the symmetric group, too. In fact, \cite{bormashenko2011coupling} uses a connection between the two walks to bound the mixing time of the transposition walk on the symmetric group, based on a coupling argument for the split-merge walk on partitions. However, this approach does not seem to be applicable to the multi-step coarse Ricci curvature.

\subsection{Glauber dynamics on statistical physical models}\label{SecStatfiz}
In this section, we are going to estimate the coarse Ricci curvature of the Glauber dynamics (with random, and systemic scan) on statistical physical models. A common property of these models is that we have some random variables (spins) $X_1, X_2,\ldots, X_N$, that are dependent on each other, and the strength of their dependence is influenced by a parameter $\beta$ (inverse temperature).

In the following, in Section \ref{SectionBoundsusingDobrushin}, first we define the Dobrushin interdependence matrix (a way to measure the strength of dependence between the random variables), and then state propositions that estimate $\kappa_k$ in terms of this matrix in the case of Glauber dynamics. In Sections \ref{CurieWeissbound} and \ref{Sec1DIsing}, we apply our results to the Curie-Weiss, and the one dimensional Ising models.

\subsubsection{Bounds using the Dobrushin interdependence matrix}\label{SectionBoundsusingDobrushin}
The following definition originates from \cite{Dobrushin1} and \cite{Dobrushin2}.
\begin{definition}[Dobrushin interdependence matrix]\label{Exgendobcond}
Let $(\Lambda,d_{\Lambda})$ be a Polish metric space (of a single spin). Define
$\Omega:=\Lambda^{N}$, and for $x,y\in \Omega$, define $d(x,y)=\sum_{i=1}^N d_{\Lambda}(x_i,y_i)$, where $x_i$ denotes coordinate $i$ of $x$.

For $x\in \Omega$, denote $x_{-i}:=(x_1,\ldots,x_{i-1},x_{i+1},\ldots,x_N)$. Given a $\Omega$ valued random vector $X=(X_1,\ldots,X_N)$ with distribution $\mu$, we say that a matrix $A:=(a_{ij})_{i,j\le N}$ is its \emph{Dobrushin interdependence matrix} if $a_{ii}=0$ for $i\le N$, and for any $x,y\in \Omega$,
\begin{equation}W_1(\mu_i(\cdot|x_{-i}),\mu_i(\cdot|y_{-i}))\le \sum_{j=1}^n a_{i,j} d_{\Lambda}(x_j,y_j).\end{equation}
Here $\mu_i(\cdot|x_{-i})$ denotes the conditional distribution of the $X_i$ given $X_{-i}=x_{-i}$,
and $W_1$ denotes the Wasserstein distance with respect to the distance $d_{\Lambda}$.
Finally, we say that $\mu$ satisfies the \emph{Dobrushin condition} if $\|A\|_1<1$.
\end{definition}
\begin{remark}A frequently used special case of this is when $d_{\Lambda}(x_i,y_i)=\II[x_i\ne y_i]$, then $W_1(\mu_i(\cdot|x_{-i}),\mu_i(\cdot|y_{-i}))$ corresponds to the total variational distance. For examples using other types of distances, see \cite{Liming}.
\end{remark}

\begin{proposition}[Glauber dynamics with random scan]\label{GlauberDobrushinProp}
Let $(\Omega,d)$, $\mu$ and $X$ and $A$ be as in Definition \ref{Exgendobcond}. Consider the Glauber dynamics Markov chain on $\Omega$ as follows. In each step, we choose a coordinate $I$ uniformly from $[N]$, and then replace $X_I$ with a conditionally independent copy, given $X_{-I}$.  Then for this Markov chain, we have
\begin{equation}\kappa_k\ge 1-\left\|\left(\frac{N-1}{N}\mtx{I}+\frac{1}{N} \mtx{A}\right)^k\right\|_1.\end{equation}
This implies, in particular, that the absolute spectral gap $\gamma^*$ satisfies
\begin{equation}\label{gammaspboundeq}\gamma^*\ge 1-\mathrm{sp}\left(\frac{N-1}{N}\mtx{I}+\frac{1}{N} \mtx{A}\right)\ge \frac{1-\mathrm{sp}(\mtx{A})}{N},
\end{equation}
where $\mathrm{sp}(\mtx{A})$ denotes the spectral radius of $\mtx{A}$.
\end{proposition}
\begin{remark}
Notice that $\left\|\left(\frac{N-1}{N}\mtx{I}+\frac{1}{N} \mtx{A}\right)^k\right\|_1$ tends to 0 as $k\to \infty$ if and only if the spectral radius of $A$ is strictly smaller than 1. This follows from the Gelfand's formula, which says that the spectral radius of a matrix $M$ equals $\lim_{k\to \infty}\|M^k\|^{1/k}$, for any induced matrix norm. This is a less restrictive criteria than $\|A\|_1< 1$. In particular, $\|A\|_{\infty}<1$, or $\|A\|_{2}<1$ also suffices.
See \cite{Liming} for a spectral gap bound for Markov processes that is similar to \eqref{gammaspboundeq}.
\end{remark}
\begin{proof}
The proof is similar to the proof of Theorem 4.3 of \cite{Cth}. We start by defining a coupling of $\Omega$ valued random variables $(X^k,Y^k)_{k\in \N}$, satisfying that $X^k\sim P_x^k, Y^k\sim P_y^k$.  First, let $X^0=x$, and $Y^0=y$. Suppose that we have already defined $(X^k,Y^k)_{0\le k\le r}$. Then let $I_r$ be uniformly distributed in $[N]$. Now we define $X^r$ and $Y^r$ as equal to $X^{r-1}$ and $Y^{r-1}$ except in their $I_r$th component. We define $X^r_{I_r}$ and $Y^r_{I_r}$ as the coupling that minimises the Wasserstein distance of the distributions $\mu_{I_r}(\cdot|X^{r-1}_{-I_r}), \mu_{I_r}(\cdot|Y^{r-1}_{-I_r})$ (if the minimising distribution does not exist, then we can make a limiting argument). 
For this coupling, define the vectors $(\mtx{l}^k)_{k\ge 0}$ taking values in $\R^N$ as
$\mtx{l}^k_i:=\E(d_{\Lambda}(X^k_i,Y^k_i))$. Using the definition of the Dobrushin interdependence matrix, we can show that for $k\ge 0$,
\[\mtx{l}^{k+1}\le \left(\frac{N-1}{N}\mtx{I}+\frac{1}{N}\mtx{A}\right)\mtx{l}^{k},\]
where the inequality is meant in each component. From this, we can see that
\[\mtx{l}^{k}\le \left(\frac{N-1}{N}\mtx{I}+\frac{1}{N}\mtx{A}\right)^k\mtx{l}^{0},\]
which implies that
$1-\kappa_k\le \left\|\left(\frac{N-1}{N}\mtx{I}+\frac{1}{N}\mtx{A}\right)^k\right\|_1$. Finally, \eqref{gammaspboundeq} follows from Gelfand's formula, and \eqref{gammastareq1}.
\end{proof}

\begin{proposition}[Glauber dynamics with systemic scan]\label{GlauberSystemicProp}
Let $\Omega$, $\mu$, $X$, and $A$ be as in Definition \ref{Exgendobcond}.
Consider a Markov chain such that in each step, we go through $X_1,\ldots,X_n$ in a row, and 
replace them with a conditionally independent copy given the rest. 
For $1\le i\le N$, define $B_i$ as a matrix equal to the identity matrix, except its $i$th row, which is the same as the $i$th row of $A$. Let $B=B_n\cdot B_{n-1}\cdot \ldots \cdot B_1$.
Then for $k\ge 1$,
\[\kappa_k\ge 1-\|B^k\|_1.\]
\end{proposition}
\begin{remark}
Similarly to the random scan case, $\|B^k\|_1\to 0$ as $k\to \infty$ if and only if the spectral radius of $B$ is less than 1.
\end{remark}
\begin{remark}
\cite{Dyersystemic} contains an estimation of the mixing time of the systemic scan Glauber dynamics under various forms of the Dobrushin condition. In particular, in Section 7 it is proven that for any $x,y\in \Omega$,
\[\dtv(P_x^k,P_y^k)\le N \|A\|_1^k,\]
implying that 
\[\tmix(\epsilon)\le 1+\frac{\log(N)+\log(1/\epsilon)}{\log(1/\|A\|_1)}\le 1+\frac{\log(N)+\log(1/\epsilon)}{1-\|A\|_1}.\]
\end{remark}
\begin{proof}[Proof of Proposition \ref{GlauberSystemicProp}]
The proof is similar to the proof of Proposition \ref{GlauberDobrushinProp}, but this time we need to show that $\mtx{l}^{k+1}\le \mtx{B} \mtx{l}^{k}$. The details are left to the reader.
\end{proof}

\subsubsection{Curie-Weiss model}\label{CurieWeissbound}
Let $\Lambda:=\{-1,1\}$, $\Omega=\Lambda^N$.
The natural distance on $\Lambda$ is $d_{\Lambda}(a,b):=\II[a\ne b]$, which induces the \emph{Hamming distance} $d(x,y):=\sum_{i=1}^{n}\II[x_{i}\ne y_{i}]$ for $x,y\in \Omega$.
For any $\omega\in \Omega$, let the Hamiltonian function be
\[H_{CW}^{\beta,h}(\omega):=\frac{\beta}{N}\cdot \sum_{1\le i<j\le N}\omega_i\omega_j+h\sum_{i\le N}\omega_i.\] 
Here $\beta>0$ is called the inverse temperature, and $h$ is the external field.
Define the probability distribution on $\Omega$ as
\begin{equation}\label{eqCWprobdef}
\pi_{CW}^{\beta,h}(\omega)=\exp(H_{CW}^{\beta,h}(\omega))/Z_{CW}^{\beta,h},
\end{equation}
where $Z_{CW}^{\beta,h}$ is a normalising constant.
In the zero magnetisation case ($h=0$), this model is known to undergo phase transition at $\beta=1$. We call $\beta<1$ the high-temperature phase, $\beta=1$ the critical phase, and $\beta>1$ the low-temperature phase.

When applying the Glauber dynamics chains (with random, or systemic scan) of the previous section to this model (see Propositions \ref{GlauberDobrushinProp} and \ref{GlauberSystemicProp}), the distribution $\pi_{CW}^{\beta,h}$ arises as their stationary distribution.
The following proposition estimates the multi-step coarse Ricci curvature of these chains.
\begin{proposition}[Ricci curvature for the Curie-Weiss model]
For the Curie-Weiss model described above, for any $h$ and $\beta$, for any $k\ge 2$, we have
\begin{align*}\kappa^{Gl.rand.scan.}&\ge \left(1-\beta\frac{N-1}{N}\right)\frac{1}{N}, \hspace{2mm}
\kappa^{Gl.sys.scan.}\ge 2-\econst^{\beta}, \\
\kappa_k^{Gl.sys.scan.}&\ge 1-\beta \econst^{\beta} \left(\beta \frac{N-1}{N}\right)^{k-1}, \hspace{2mm}\gammaps^{Gl.sys.scan.}\ge \frac{1-\beta\cdot (N-1)/N}{4}.
\end{align*}
Finally, for $\beta=1$ and $h=0$ (the critical phase), there exists a universal constant $C>0$ such that for any $N$, any $k\ge CN^{3/2}\log(N)$, $\kappa^{Gl.rand.scan.}_{k}\ge 1/2$, and $\left(\kappa_{\Sigma}^c\right)^{Gl.rand.scan.}\le 2CN^{3/2} \log(N)$.
\end{proposition}
\begin{proof}
A simple calculation shows that for the Curie-Weiss model, the following matrix is a Dobrushin interdependence matrix for any $\beta$ and $h$ (albeit not the sharp one for $h\ne 0)$. 
\[A^{CW}:=\left( \begin{array}{ccccccc}
0 & \beta/N & \beta/N & \beta/N &  \ldots \\
\beta/N &  0 & \beta/N & \beta/N&  \ldots \\
\vdots &\vdots & \vdots &\vdots  & \ldots\\
\beta/N &\beta/N & \beta/N & \ldots & 0\\
\end{array} \right).\]
Since $\|A^{CW}\|_1<\beta(N-1)/N$, $\kappa^{Gl.rand.scan.}\ge 1-\beta(N-1)/N$ by Proposition \ref{GlauberDobrushinProp}. For the Glauber dynamics with systemic scan, we apply Proposition \ref{GlauberSystemicProp} with the Dobrushin interdependence matrix $A^{CW}$.
Let $x:=\beta/N$, then after some calculations, we obtain that the matrix $B$ is given by
\begin{align*}b_{i,1}&=0, \hspace{2mm} b_{i,i+1}=b_{i,i+2}=\ldots=b_{i,N}=x(1+x)^{i-1}, \\
b_{i+k,i}&=x\cdot\left((1+x)^{i+k-1}-(1+x)^k\right),
\end{align*}
for $1\le i\le N$, $0\le k\le N-i$. Now for any $k\ge 1$, $\|B^k\|_1=\max(\mtx{1} \cdot B^k)$ (maximum column sum), with $\mtx{1}$ denoting a row vector of ones, and $\max$ denoting the maximal element of the vector. After a simple calculation, we get that for $1\le i\le N$,
\[(\mtx{1}\cdot B)_i=(1+x)^N-(1+x)^{N-i+1},\]
which implies that $\|B\|_1=\max(\mtx{1}\cdot B)=(1+x)^N-1-x=\econst^{\beta}-1-\beta/N$, thus by Proposition \ref{GlauberSystemicProp},
\[\kappa^{Gl.sys.scan.}\ge 2-\econst^{\beta}.\]

As we can see, $\kappa$ is negative for part of the high temperature case ($\beta<1$).
Now we will use the following lemma.
\begin{lemma}
Let $v=(0,1/(N-1),2/(N-1),\ldots,1)$. Then for $B$ defined as above, 
\[(v\cdot B)_{i}\le v_{i}\cdot \left(\beta \cdot \frac{N-1}{N}\right).\]
\end{lemma}
This lemma can be proven by straightforward calculations, which we omit.

Now it is easy to see that for $1\le i\le N$, 
\[(\mtx{1}\cdot B)_i\le ((1+x)^N-(1+x)^{N-1})\cdot (i-1)=\beta (1+x)^{N-1}\cdot \frac{i-1}{N}\le \beta \frac{N-1}{N} \econst^{\beta}  \frac{i-1}{N-1},\]
and thus by the above lemma, we can conclude that 
\[\|B^k\|_1=\max(\mtx{1}\cdot B^k)\le  \econst^{\beta} \left(\beta \frac{N-1}{N}\right)^{k},\]
which implies that for $k\ge 1$,
\[\kappa_k\ge 1-\|B^k\|_1\ge 1- \econst^{\beta} \left(\beta \frac{N-1}{N}\right)^{k}.\]
From this, for $0\le \beta\le 1$, with the choice of $k=\lceil 2/(1-\beta\cdot (N-1)/N)\rceil$ (using the identity $(1-c)^{(1/c)}\le (1/e)$ for $c>0$), we get $\kappa_k\ge 1-1/e$.
By symmetry $\kappa_k(\mtx{P}^*)=\kappa_k$, so using \eqref{gammapseq1}, we get
\[\gammaps^{Gl.sys.scan.}\ge (1-1/\econst^2)/\lceil 2/(1-\beta\cdot (N-1)/N)\rceil\ge \frac{1-\beta\cdot (N-1)/N}{4}.\]
Finally, we move to the case of the critical phase ($\beta=1, h=0$). Theorem 2 of\cite{LevinLuczakPeres} (see also \cite{DingMeanfieldIsing}) shows that the mixing time satisfies $\tmix=\mathcal{O}(N^{3/2})$, thus \eqref{kappatmixbound} gives us the bound on $\kappa_k$, and by \eqref{eqkappasigmacbound}, we get the bound on $\kappa_{\Sigma}^c$.
\end{proof}

Substituting the bound $\left(\kappa_{\Sigma}^c\right)^{Gl.rand.scan.}\le 2CN^{3/2} \log(N)$ and $\hat{\sigma}_{\max}=1$ to Theorem \ref{revconcthm} leads to the following concentration inequality (a new result).
\begin{proposition}[Concentration for the critical phase]
In the critical phase of the Curie-Weiss model ($\beta=1, h=0$), for any $f:\Omega\to \R$ that is 1-Lipschitz with respect to $d$ (Hamming distance), for any $t\ge 0$,
\[\PP(|f(X)-\E f|\ge t)\le 2\exp\left(-\frac{t^2}{2CN^{3/2}\log(N)}\right),\]
where $X\sim \pi_{CW}^{1,0}$, and $C$ is an universal constant.
\end{proposition}
\begin{remark}
This most likely holds without the $\log(N)$ term as well. The constant in the exponent should be at least of order $N^{3/2}$, as one can see from the limiting distribution of the magnetisation ($f(\omega)=\sum_{i=1}^N \omega_i$), where one has to normalise by $N^{3/4}$ (see \cite{ChatterjeeShao}, page 466).  Proposition 4 of \cite{ChatDey} shows a subgaussian ($\exp(-c t^4)$) concentration bound for the magnetisation.
\end{remark}

\subsubsection{1D Ising model}\label{Sec1DIsing}
Let $\Omega=\{-1,1\}^N$. Let $d$ be the Hamming distance on $\Omega$, as in the previous section. For any $\omega\in \Omega$, let the \emph{Hamiltonian} function be
\[H_{\mathrm{I1D}}^{\beta,h}(\omega):=\frac{\beta}{N}\cdot \sum_{1\le i<N}\omega_i\omega_{i+1}+h\sum_{i\le N}\omega_i.\] 
Here $\beta>0$ is called the inverse temperature, and $h$ is the external field.
Define the probability distribution on $\Omega$ as
\begin{equation}\label{eqCWprobdef}
\pi_{1D}^{\beta,h}(\omega)=\exp(H_{\mathrm{I1D}}^{\beta,h}(\omega))/Z_{\mathrm{I1D}}^{\beta,h},
\end{equation}
where $Z_{\mathrm{I1D}}^{\beta,h}$ is a normalising constant.
This model is known to have no phase transition. The following proposition applies our results on this model, assuming that $h=0$.
\begin{proposition}[Ricci curvature for 1D Ising model]
For the 1D Ising model described above, for $h=0$, for any $\beta>0$, let $\rho:=1/(1+\econst^{-4\beta})$, then
\begin{align*}\kappa^{Gl.rand.scan.}&\ge \frac{2}{N}(1-\rho), 
\hspace{2mm}\kappa^{Gl.sys.scan.}\ge 2(1-\rho)/(3/2-\rho), \\
\gammaps^{Gl.sys.scan.}&\ge 2(1-\rho)/(3/2-\rho)^2.
\end{align*}
\end{proposition}
\begin{proof}
For the 1 dimensional Ising model, the probability of a spin being $1$, given that $m$ of it's neighbours are 1, $m=0,1,2$, is
\[\frac{1}{1+\exp(4\beta-2h)}, \frac{1}{1+\exp(-2h)}, 
\frac{1}{1+\exp(-4\beta-2h)},\text{ respectively.}\]
It follows that for this model, the Dobrushin matrix is tridiagonal, with the diagonal elements being 0. For $h\le 0$, the above and below-diagonal elements equal
$\frac{1}{1+\exp(-4\beta-2h)}-\frac{1}{1+\exp(-2h)}$, while for $h> 0$, they equal $\frac{1}{1+\exp(-2h)}-\frac{1}{1+\exp(4\beta-2h)}$.  In the case of zero external field, $h=0$, the upper and lower diagonal elements equal $\rho-1/2$, and $\|A^{\mathrm{I1D}}\|_1=2\rho-1<1$. Using this, $\kappa^{Gl.rand.scan.}\ge \frac{2}{N}(1-\rho)$ follows by Proposition \ref{GlauberDobrushinProp}.

In the systemic scan case, it is easy to see that for $1\le j\le N$, $b_{j1}=0$, for $1<r\le N$,
$b_{{r-1},r}=\rho-1/2$, $b_{r,r}=(\rho-1/2)^2$, and for $r<j\le N$, $b_{j,r}=(\rho-1/2)^{2+j-r}$. This impies that $\|B\|_1\le (\rho-1/2)/(1-\rho+1/2)$, and $\kappa^{Gl.sys.scan.}\ge 2(1-\rho)/(3/2-\rho)$ follows by Proposition \ref{GlauberSystemicProp}. Finally, by symmetry and using Proposition \ref{kappagammaprop}, we have \[\gammaps^{Gl.sys.scan.}\ge 1-(1-\kappa^{Gl.sys.scan.})^2\ge 2(1-\rho)/(3/2-\rho)^2.\qedhere\]
\end{proof}

\subsection{Random walk on a binary cube with a forbidden region}\label{SecBinaryCube}
Consider a binary cube $\Omega_0:=\{0,1\}^N$. We call the region $F:=\{x\in \Omega_0, \sum x_i< R\}$  the forbidden region. Let $\Omega:=\Omega_0\setminus F$.
We consider the following random walk (a version of Glauber dynamics) on $\Omega$. If we are in $x$, then we pick an index $I$ out of $\{1,\ldots,N\}$ uniformly, and 
\begin{itemize}
\item if $\sum_{i=1}^{N} x_i>R$, or if $\sum_{i=1}^{N} x_i=R$ and $x_I=0$, then $x_I$ is replaced with an independent Bernoulli(1/2) random variable,
\item if $\sum_{i=1}^{N} x_i=R$, and $x_I=1$, then we do nothing, and stay in $x$.
\end{itemize}
The stationary distribution $\pi$ is the uniform distribution on $\Omega$ (the random walk can be shown to be reversible with respect to this distribution). Because of the geodesic property, it is sufficient to look at $\kappa_k(x,y)$ for neighbouring $x$ and $y$. 
Because of symmetry, we can denote this by $\kappa_k(j):=\kappa_k(x,y)$ for $x$ such that $\sum_{i=1}^N x_i=j$, $y$ such that $\sum_{i=1}^{N} y_i=j+1$, and $\sum_{i=1}^{N} \II[x_i\ne y_i]=1$ (for $R\le j\le N-1$). Initially, we have negative curvature,
\[\kappa_1(R)=\frac{2-R}{2N}, \kappa_1(j)=\frac{1}{N} \text{ for }R< j\le N-1.\]
From Proposition \ref{kstepricciestimateprop}, we get the recursive bounds
\begin{align*}
\kappa_{k+1}(R)&\ge \frac{2-R}{2N}+\frac{N-1+R}{2N}\cdot \kappa_k(R)+\frac{N-R-1}{2N}\cdot \kappa_k(R+1),\\
\kappa_{k+1}(j)&\ge \frac{1}{N}+\frac{N-1}{2N}\cdot \kappa_k(j)+\frac{j}{2N}\cdot \kappa_k(j-1)+\frac{N-j-1}{2N}\cdot \kappa_k(j+1)
\end{align*}
for $R< j\le N-1$. Notice that all the coefficients of $\kappa_{k}(j)$ in these inequalities are positive. This implies that if we let $\tilde{\kappa}_1(j):=\kappa_1(j)$ for $R\le j\le N-1$, and let
\begin{align*}
\tilde{\kappa}_{k+1}(R)&:=\frac{2-R}{2N}+\frac{N-1+R}{2N}\cdot \tilde{\kappa}_k(R)+\frac{N-R-1}{2N}\cdot \tilde{\kappa}_k(R+1),\\
\tilde{\kappa}_{k+1}(j)&:= \frac{1}{N}+\frac{N-1}{2N}\cdot \tilde{\kappa}_k(j)+\frac{j}{2N}\cdot \tilde{\kappa}_k(j-1)+\frac{N-j-1}{2N}\cdot \tilde{\kappa}_k(j+1)
\end{align*}
for $R< j\le N-1$, then $\kappa_k(j)\ge \tilde{\kappa}_k(j)$ for every $k\ge 1$, $R\le j\le N-1$, implying that
$\tilde{\kappa}_k:=\min_{R\le j\le N-1}\tilde{\kappa}_k(j)\le \kappa_k$.
It is easy to conduct numerical simulations to see the behaviour of this recursion. The figures below show this for $N=500$, $R=100$.
\begin{figure}
\centering
\begin{tabular}{cc}
\addtolength{\subfigcapskip}{0.2cm}
\subfigure[$\tilde{\kappa}_1(j)$ for $R\le j<R+30$]{
    \includegraphics[width = 5cm]{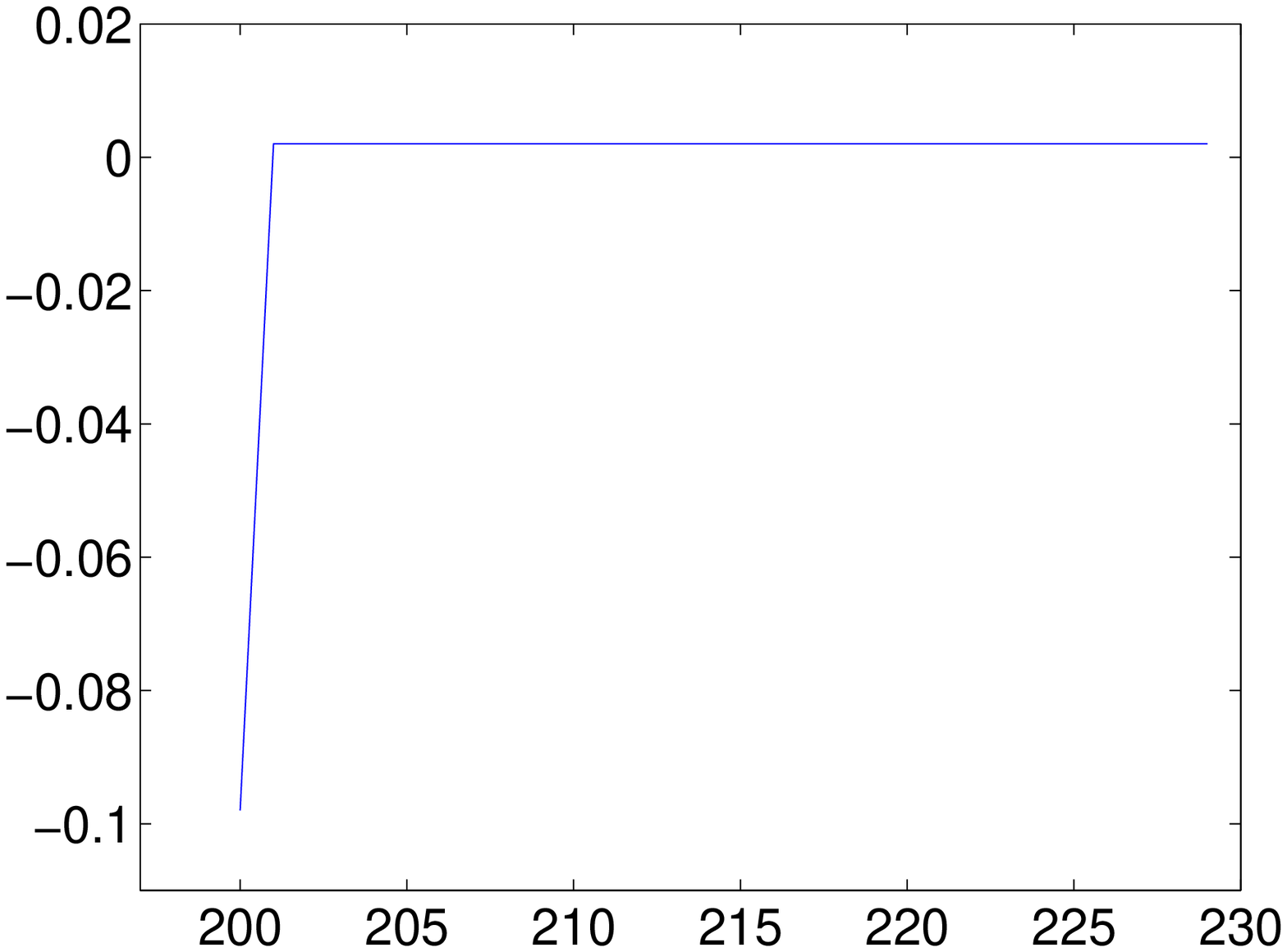}
	}
&
\addtolength{\subfigcapskip}{0.2cm}
\subfigure[$\tilde{\kappa}_{100}(j)$ for $R\le j<R+30$]{
    \includegraphics[width = 5cm]{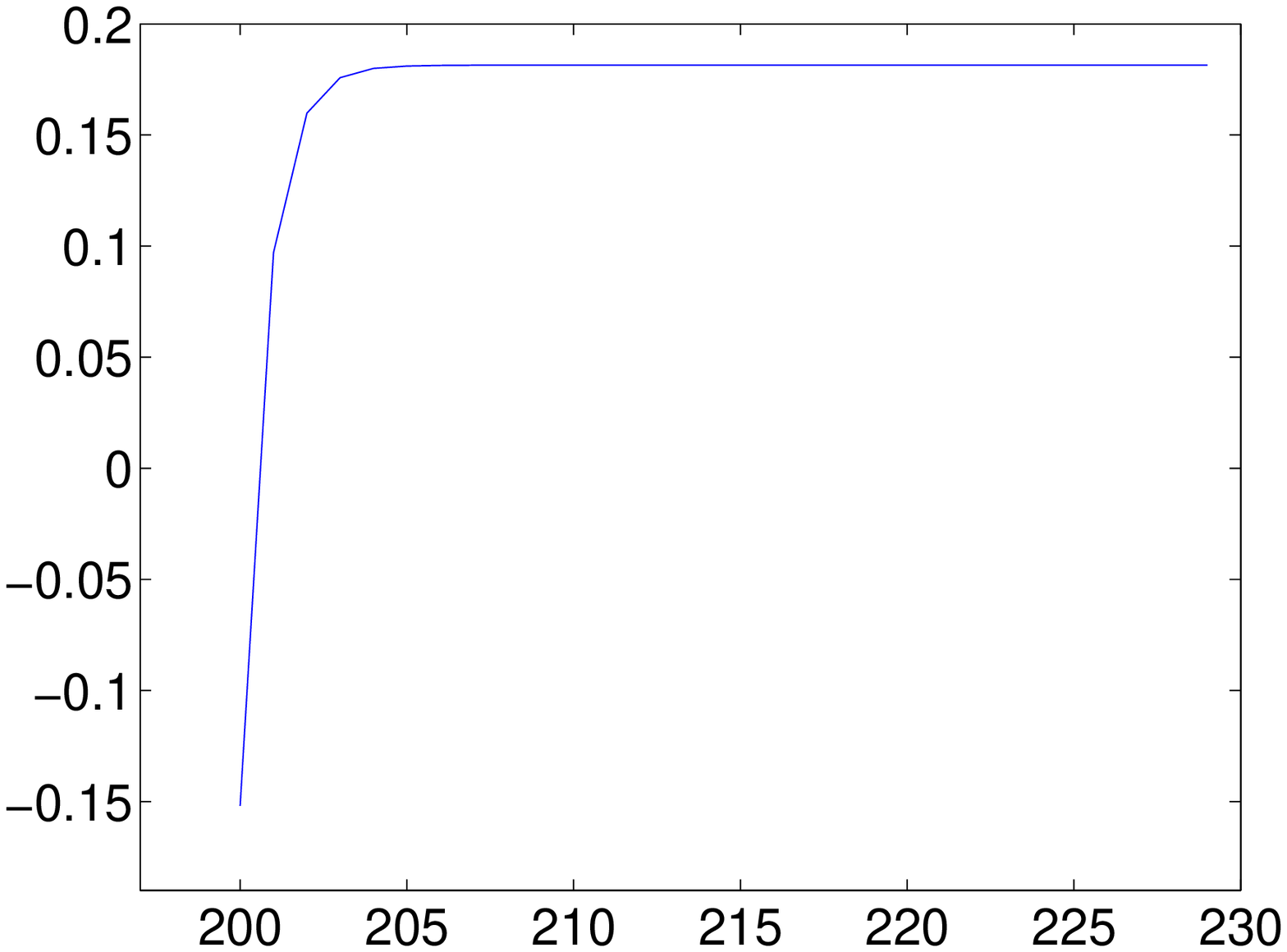}
	}
\\

\addtolength{\subfigcapskip}{0.2cm}
\subfigure[$\tilde{\kappa}_{500}(j)$ for $R\le j<R+30$]{
    \includegraphics[width = 5cm]{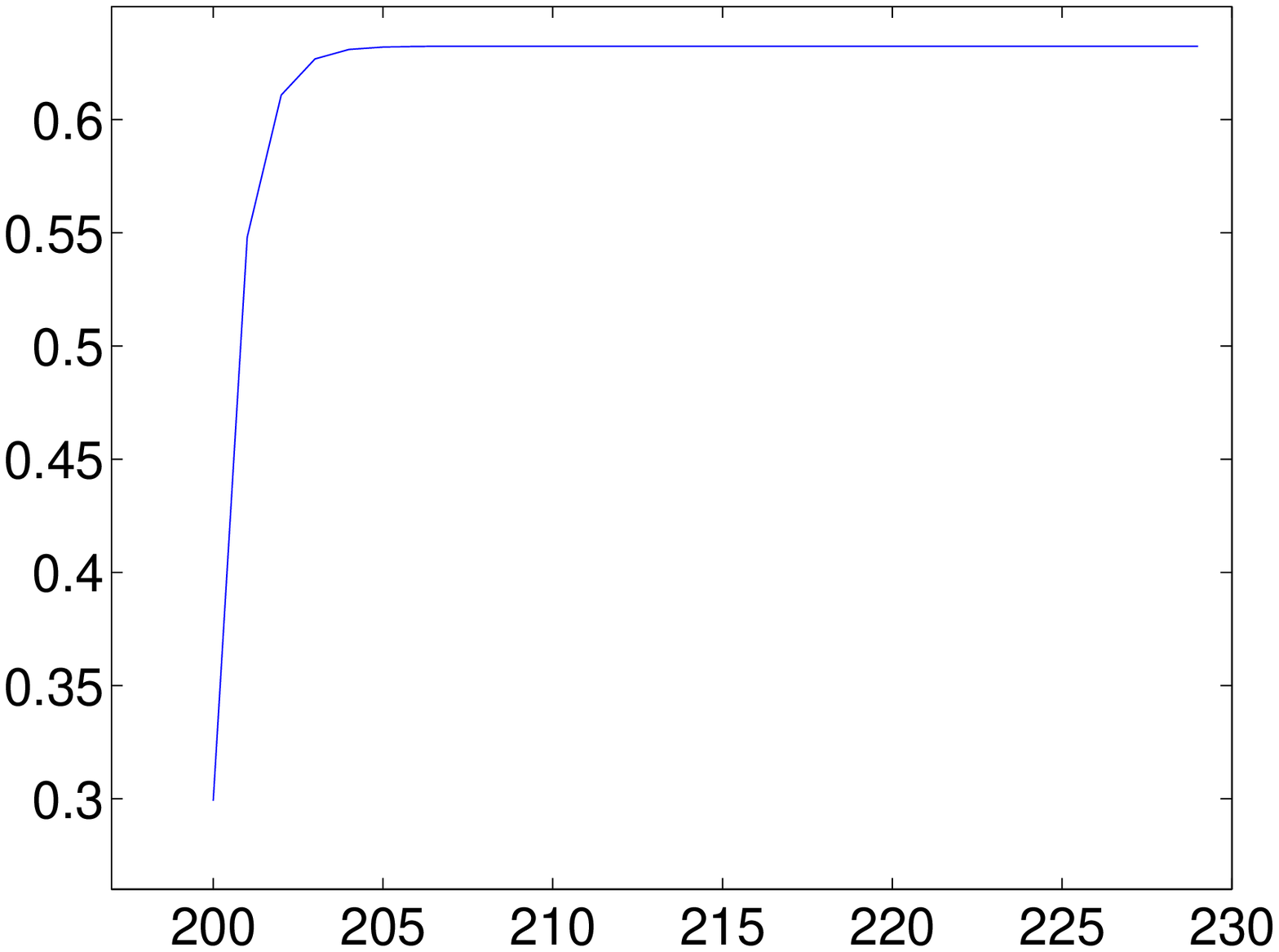}
	}
&	
\addtolength{\subfigcapskip}{0.2cm}
\subfigure[$\tilde{\kappa}_k$ for $1\le k\le 500$]{
    \includegraphics[width = 5cm]{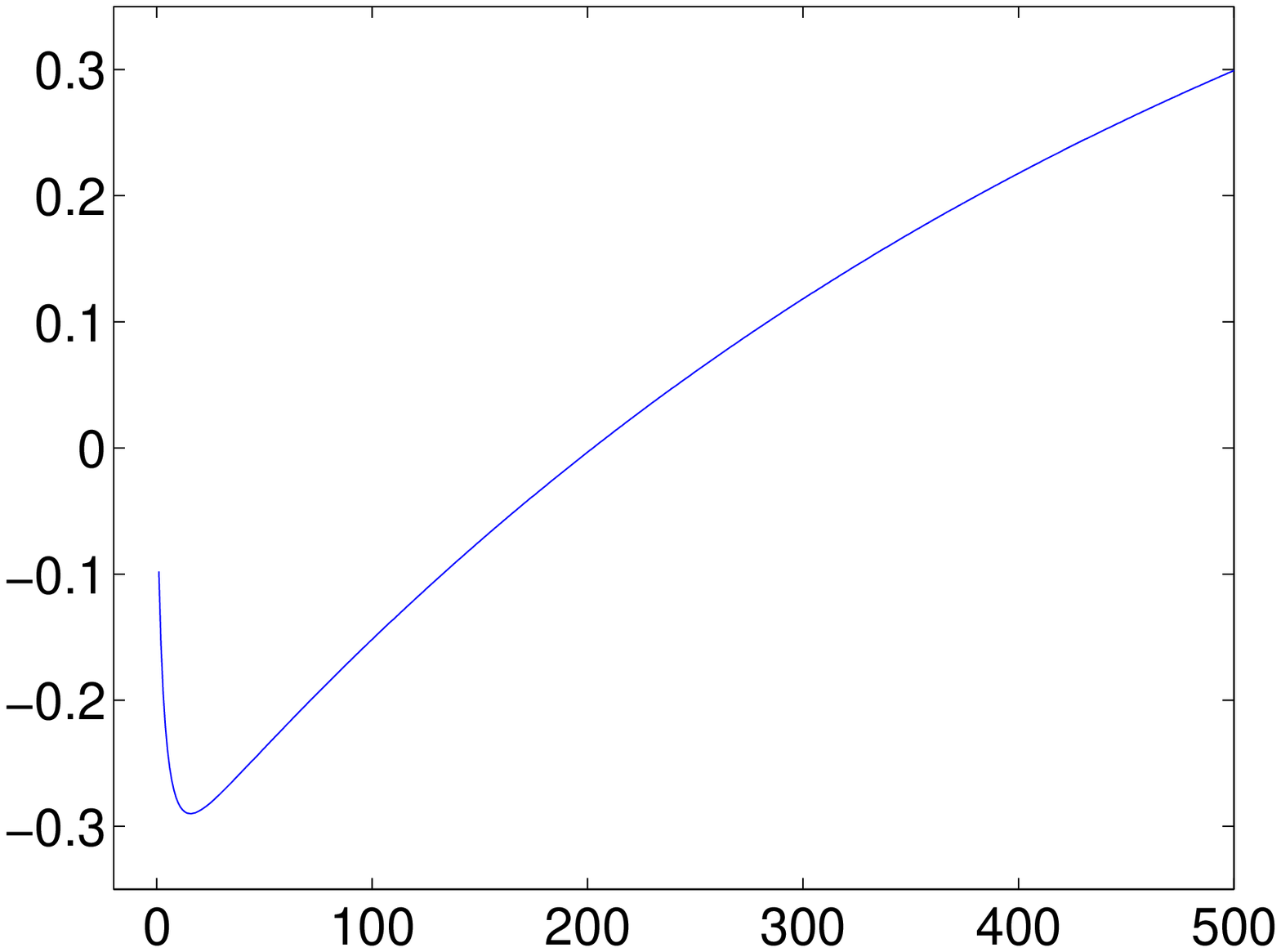}
	}
\end{tabular}
\caption[Evolution of the multi-step coarse Ricci curvature]{Evolution of the multi-step coarse Ricci curvature}
\end{figure}
The figures show that initially $\hat{\kappa}_k$ is decreasing, and stays negative, but eventually the positive curvature wins, and $\hat{\kappa}_k$ becomes positive. The following proposition gives bounds on $\kappa_N$ and $\kappa_{\Sigma}^c$ based on this recursion.
\begin{proposition}[Curvature bound for binary cube with forbidden region]\label{binarykappaprop}
Let $\rho(R/N):=\frac{1}{2}-\frac{1}{e}-\frac{R}{N-2R}$. Then for $R\le N/10$,
\[\kappa_N\ge \rho(R/N), \text{ and }\kappa_{\Sigma}^c\le \frac{N}{\rho(R/N)}\cdot \left(1+\frac{R}{N-2R}\right).\]
\end{proposition}
\begin{remark}
By Propositions \ref{mixingRicciprop} and \ref{kappagammaprop}, the spectral gap and mixing time of the walk can be bounded as $\gamma\ge \frac{1}{N}\rho(R/N)$, $\tmix\le 2N\log(N)/\rho(R/N)$.
Moreover, by Theorem \ref{revconcthm}, it follows that for a random vector $X\sim \pi$ and for any $1$-Hamming-Lipschitz function $f$, for any $t\ge 0$,
\[\PP(|f(X)-\E(f)|\ge t)\le 2\exp\left(-\frac{t^2}{N}\cdot \rho(R/N)/\left(1+\frac{R}{N-2R}\right)\right).\]
\end{remark}
\begin{proof}[Proof of Proposition \ref{binarykappaprop}]
Let $\epsilon:=R/N$, and for $0\le i\le N-R-1$, $k\ge 1$, let
\begin{equation}
\hat{\kappa}_k(R+i):=-\frac{\epsilon}{1-2\epsilon}\cdot \left(\frac{\epsilon}{1-\epsilon}\right)^i+\left(1-\exp\left(-\frac{k}{N}\right)\right)-\frac{k-1}{2N}.
\end{equation}
Then it is easy to see that $\hat{\kappa}_1(j)\le \tilde{\kappa}_1(j)\le \kappa_1(j)$ for $R\le j\le N-1$. Moreover, one can verify that for every $k\ge 1$,
\begin{align*}
\hat{\kappa}_{k+1}(R)&\le \frac{2-R}{2N}+\frac{N-1+R}{2N}\cdot \hat{\kappa}_k(R)+\frac{N-R-1}{2N}\cdot \hat{\kappa}_k(R+1),\\
\hat{\kappa}_{k+1}(j)&\le \frac{1}{N}+\frac{N-1}{2N}\cdot \hat{\kappa}_k(j)+\frac{j}{2N}\cdot \hat{\kappa}_k(j-1)+\frac{N-j-1}{2N}\cdot \hat{\kappa}_k(j+1)
\end{align*}
for $R< j\le N-1$, implying that $\hat{\kappa}_k(j)\le \tilde{\kappa}_k(j)\le \kappa_k(j)$. The bound on $\kappa_N$ now follows by noticing that $\kappa_k\ge \hat{\kappa}_k(R)$ for every $k\ge 1$, and the bound on $\kappa_{\Sigma}^c$ follows from \eqref{eqkappasigmacbound}.
\end{proof}

\section{Proofs of concentration results}\label{SecProofs}
In this section, we present the proofs of our concentration inequalities. First, we briefly review Chatterjee's method of proving concentration inequalities via Stein's method of exchangeable pairs. We prove our variance and concentration bounds for reversible chains using this approach. Finally, we prove our variance and concentration bounds without using reversibility, by a modification of Ollivier's proofs.
\subsection{Concentration inequalities via the method of exchangeable pairs}
For the proof of our theorems about reversible chains, we will use Stein's method of exchangeable pairs for concentration inequalities, developed in \cite{Cth}. Let $(X,X')$ be an exchangeable pair taking values in a Polish space $\Omega$. Let $f:\Omega\to \R$, $\E f(X)=0$, $\E (f(X)^2)<\infty$. Suppose that there is an antisymmetric function $F:\Omega^2 \to \R$ such that $\E (F(X,X')|X)=f(X)$. Define 
\begin{equation}
\Delta(X):=\frac{1}{2}\E(|F(X,X') (f(X)-f(X'))| |X),
\end{equation}
and assume that $\Delta(X)<\infty$ almost surely. Then the following results hold.

\begin{thm}[Theorem 3.2 of \cite{Cth}]\label{thm32}
With the above notations,
\[\Var (f(X))=\frac{1}{2}\E((f(X)-f(X'))F(X,X')).\]
\end{thm}
\begin{thm}[Theorem 3.14 of \cite{Cth}]\label{thm314}
For any positive integer $p$, we have
\[\E\left([f(X)-\E(f)]^{2p}\right)\le (2p-1)^p\E(\Delta(X)^p).\]
\end{thm}

\begin{thm}[Theorem 3.3 of \cite{Cth}]\label{thm33}
If $\Delta(X)\le C$ almost surely, then for any $\theta\in \R$,
$\E(\econst^{\theta f(X)})\le \exp(\theta^2 C/2)$, and
\[\PP(|f(X)-\E (f(X))|\ge t)\le 2\exp\left(\frac{-t^2}{2C}\right).\]
\end{thm}

\begin{thm}[Theorem 3.13 of \cite{Cth}]\label{thm313}
Let $r(L):= \frac{\log \E(\econst^{L\Delta(X)})}{L}$. Then for any $L>0$ such that $r(L)<\infty$, we have
\[\PP(|f(X)-\E (f(X))|\ge t)\le 2\exp\left(\frac{-t^2}{2r(L)+4tL^{-1/2}}\right).\]
\end{thm}

Now we show how to find $F(x,y)$ for a given $f(x)$ (based on Section 4 of \cite{Cth}).
First, notice, that an exchangeable pair $(X,X')$ induces a Markov kernel $P$, defined as
\[P(x,A):=\PP(X'\in A|X=x) \text{ for every }x\in \Omega, \text{ and every measurable }A\subset \Omega.\]
Conversely, for a reversible Markov kernel $P$ on $\Omega$ with stationary distribution $\pi$, we define an exchangeable pair as $X\sim \pi$, and $\PP(X'\in A|X=x):=P(x,A)$.
The following lemma explains the construction of $F(x,y)$ (this is a straightforward extension of Lemma 4.1 of \cite{Cth}).
\begin{lemma}[Construction of $F(x,y)$ from $f(x)$]\label{Fxylemma}
Let $X$, $X'$ and $P$ as above. Let $f:\Omega\to \R$ be a measurable function with $\E (f(X))=0$. Suppose that for every $x,y \in \Omega$, there is a constant $L(x,y)< \infty$ such that $L(y,x)=L(x,y)$,
\begin{equation}\label{eqlemmaLcond}\sum_{k=0}^{\infty} |P^k f(x)-P^k f(y)|\le L(x,y), \text { and that }\E(L(X,X')|X)<\infty \text{ almost surely}.\end{equation}
Then the function
\begin{equation}\label{Fxyeq}
F(x,y)=\sum_{k=0}^{\infty}(P^k f(x)-P^k f(y))
\end{equation}
satisfies $F(x,y)=-F(y,x)$ and $\E(F(X,X')|X)=f(X)$.
\end{lemma}
\begin{proof}
We have $\E(P^k f(X)-P^k f(X')|X)=P^k f(X)-P^{k+1} f(X)$, and thus
\begin{equation}\label{eqsumpk}\sum_{k=0}^{N}\E(P^k f(X)-P^k f(X')|X)=f(X)-P^{N+1}f(X).\end{equation}
Now by \eqref{eqlemmaLcond}, and Lebesgue's dominated convergence theorem, the left hand side will converge to a limit as $N\to \infty$.
For the right hand side, we have $P^{N+1}f(y)-P^{N+1}f(x)\to 0$ by \eqref{eqlemmaLcond} for any $x,y\in \Omega$. The expected value of both sides of \eqref{eqsumpk} is 0, so $\lim_{N\to \infty}P^{N+1}f(x)=0$ for every $x\in \Omega$, and the claim of the lemma follows.
\end{proof}

\subsection{Concentration of Lipschitz functions under the stationary distribution}

We start with the variance bounds.

\begin{proof}[Proof of Theorem \ref{Varianceboundthm}]
The proof of this result is similar to the proof of Proposition 
32 of \cite{Ollivier2}. Assume first that $\|f\|_{\infty}<\infty$, then $\Var(f)<\infty$. Now we are going to show that if $\kappa_{\Sigma}^c<\infty$, then $\Var(P^k f)\to 0$ as $k\to \infty$. Let $B_r$ be a ball of radius $r$ centred at some point in $\Omega$, then we can write
\[\Var (P^k f)=\frac{1}{2}\int\int(P^k f(x)-P^k f(y))\mathrm{d} \pi(x) \mathrm{d} \pi(y)\le 
2(1-\kappa_k)^2 r^2 + 2\|f\|_{\infty}^2 \pi(\Omega\setminus B_r).\]
If we set $r=(1-\kappa_k)^{-1/2}$, then this will tend to 0 as $k\to \infty$, since $\kappa_{\Sigma}^c<\infty$ implies that $1-\kappa_k\to 0$. Moreover, if $\kappa_{\Sigma}^c=\infty$, then there is nothing to prove. By simple algebra, one can show that
\[\Var (f)= \Var (P f) + \int_{x\in \Omega} \Var_{P_x}(f)\mathrm{d}\pi(x),\]
and then using the fact that $\Var(P^k f)\to 0$, we obtain
\[\Var (f)= \sum_{k=0}^{\infty} \int_{x\in \Omega} \Var_{P_x}(P^k f)\mathrm{d}\pi(x).\]
In the above summation, $\mtx{P}^k (f)$ is $(1-\kappa_k)$-Lipschitz, so by the definition of the local dimension $n(x)$, we have $\Var_{P_x}(\mtx{P}^k (f))\le (1-\kappa_k)^2 \sigma^2(x)/n(x)$, and \eqref{nonrevvareq} follows under the boundedness assumption. Finally, the $\|f\|_{\infty}=\infty$ case can be handled by a limiting argument.
\end{proof}

Now we prove concentration inequalities for the reversible case.
\begin{proof}[Proof of Theorem \ref{revconcthm}]
Based on Lemma \ref{Fxylemma} one can show that
$F(x,y)=\sum_{k=0}^{\infty}(P^k f(x)-P^k f(y))$ satisfies the antisymmetry and $\E(F(X,X')|X)=f(X)$ conditions, and thus the exchangeable pair method is applicable. With this choice of $F$, we have
\begin{align}
\nonumber\Delta(X)&=\frac{1}{2}\E\left(\left.\left|\sum_{k=0}^{\infty}(P^k f(X)-P^k f(X'))\right| \cdot |f(X)-f(X')|\, \right|X\right)\\
\label{deltaxboundsigmaxeq}&\le \frac{1}{2}\E\left(\left. \sum_{k=0}^{\infty}(1-\kappa_k(X,X')) d(X,X')^2\right|X\right)\le \frac{1}{2}\E\left(\kappa_{\Sigma}^c(X,X')d(X,X')^2|X\right).
\end{align}
The above quantity can be further bounded by $\frac{1}{2}\kappa_{\Sigma}^c \hat{\sigma}^2_{\max}$. Based on this, Theorem \ref{thm33} implies inequality \eqref{revconceq1}, and the bound
\begin{equation}\label{fmomgenineq1}
\E\left(\econst^{\theta f(X)}\right)\le \exp\left(\theta^2 \kappa_{\Sigma}^c \hat{\sigma}^2_{\max}/4\right).
\end{equation}

Now we turn to the proof of the second claim of the theorem. Let $g(x)=V(x)-\E(V(X))$. Notice that by applying \eqref{fmomgenineq1} to $g/\|V\|_{\Lip}$ in the place of $f$, we obtain that for any $L>0$,
\[\E(\econst^{L\Delta(X)})\le \E\left(\econst^{LV(X)/2}\right)\le  \econst^{L \E(V(X))/2}\cdot \econst^{L^2 \|V\|_{\Lip}^2 \kappa_{\Sigma}^c \hat{\sigma}_{\max}^2/16}.\]
By choosing $L$ as stated, we obtain that $\E(\econst^{L\Delta(X)})= \econst^{L \E(V(X))}$, and thus 
$r(L)=E(V(X))$, and inequality \eqref{revconceq2} follows by Theorem \ref{thm313}.
\end{proof}
Our next proof is the moment bound for reversible chains.
\begin{proof}[Proof of Theorem \ref{Revmomentboundthm}]
From \eqref{deltaxboundsigmaxeq}, we have $\Delta(X)\le \frac{1}{2}\kappa_{\Sigma}^c\hat{\sigma}(X)^2$, and applying Theorem \ref{thm314} leads to this result.
\end{proof}

Now we prove concentration bounds without using reversibility.

The proof of Theorem \ref{nonrevconcthm} is based on the following two lemmas (the first one is a slight variation of Lemma 38 of \cite{Ollivier2}).
\begin{lemma}\label{nonrevlemma}
Let $\varphi:\Omega \to \R$ be an $\alpha$-Lipschitz function. Assume that $\lambda\le 1/(3\sigma_{\infty})$. For $r\in \R$, let $g(r):=\econst^{(2/3)r}\cdot r^2/2$. Then for $x\in \Omega$, we have
\[(\mtx{P}\econst^{\lambda \varphi})(x)\le \exp\left(\lambda \mtx{P} \varphi(x) + \lambda^2 \frac{\sigma(x)^2}{n(x)}\cdot g(\alpha)\right).\]
\end{lemma}
\begin{proof}
The proof is a simple modification of the original argument. First, by the Taylor expansion with Lagrange remainder, it follows that for any smooth function $g$, and any real valued random variable $Y$, $\E g(Y)\le g(\E Y) + \frac{1}{2}\left(\sup g''\right)\Var Y$. For $g(Y)=\econst^{\lambda Y}$, this yields that for any $x\in \Omega$,
\[\left(\mtx{P} \econst^{\lambda \varphi}\right)(x)=\E_{P_x} \econst^{\lambda \varphi}\le \econst^{\lambda (\mtx{P} \varphi)(x)}+\frac{\lambda^2}{2}\left(\sup_{\mathrm{Supp} P_x} e^{\lambda \varphi}\right)\Var_{P_x} \varphi.\]
By Definition \ref{manydefs}, we have $\mathrm{diam Supp}P_x\le 2\sigma_{\infty}$, and using the $\alpha$-Lipschitz property of $\varphi$, 
$\sup_{\mathrm{Supp} P_x} \varphi \le \E_{P_x} \varphi+ 2\sigma_{\infty}\alpha$. Therefore, we have
\[\left(\mtx{P} \econst^{\lambda \varphi}\right)(x)\le \econst^{\lambda (\mtx{P} \varphi)(x)}+\frac{\lambda^2}{2} \econst^{\lambda (\mtx{P}\varphi)(x) +2\lambda \sigma_{\infty}\alpha}\cdot \Var_{P_x} \varphi.\]
By Definition  \ref{manydefs}, we have $\Var_{P_x} \varphi\le \alpha^2\frac{\sigma(x)^2}{n_x}$. Finally, for $\lambda\le 1/(3\sigma_{\infty})$, we have $e^{2\lambda \sigma_{\infty}\alpha}\le e^{(2/3) \alpha}$, thus
\[\left(\mtx{P} \econst^{\lambda \varphi}\right)(x)\le \econst^{\lambda (\mtx{P} \varphi)(x)}\cdot \left( 1+ \lambda^2 \frac{\sigma(x)^2}{n_x} g(\alpha)\right)\le \exp\left(\lambda \mtx{P} \varphi(x) + \lambda^2 \frac{\sigma(x)^2}{n(x)} g(\alpha)\right).\]
\end{proof}
\begin{lemma}\label{tmaxlambdamaxlemma}
Suppose that a function $f:\Omega\to \R$ satisfies that for $0\le \lambda\le \lambda_{\max}$,
\[\E(\exp(\lambda f))\le \exp(\lambda \E(f)+\lambda^2 C).\]
Let $t_{\max}:=2C\lambda_{\max}$, then for $0\le t\le t_{\max}$, we have
\[\PP(f(X)\ge \E(f)+t)\le \exp\left(-\frac{t^2}{4C}\right),\]
and for $t\ge t_{\max}$, we have
\[\PP(f(X)\ge \E(f)+t)\le \exp\left(-\frac{t_{\max}^2}{4C}-(t-t_{\max})\lambda_{\max}\right).\]
\end{lemma}
\begin{proof}
This follows by the standard Markov inequality argument.
\end{proof}

\begin{proof}[Proof of Theorem \ref{nonrevconcthm}]
Fix some $\lambda\in [0,1/(3\sigma_{\infty})]$. Let $f_0:=f$, and for $k\ge 0$, define $f_{k+1}$ as
\[f_{k+1}(x):=\mtx{P}f_k(x)+\lambda g(\|f_k\|_{\Lip}) \cdot S_{\max}.\]
Lemma \ref{nonrevlemma} shows that
$\left(\mtx{P} \econst^{\lambda f_k}\right)(x)\le \econst^{\lambda f_{k+1}(x)}$, and thus $\left(\mtx{P}^k \econst^{\lambda f}\right)(x)\le \econst^{\lambda f_{k}(x)}$.

Since $S_{\max}$ is a constant, we have $\|f_k\|_{\Lip}=\Lip(\mtx{P}^k f)\le (1-\kappa_k)\|f\|_{\Lip}$, and 
\begin{equation}\label{eqstar}f_k(x)\le \mtx{P}^k f(x)+\lambda S_{\max} \sum_{i=0}^{k-1}g((1-\kappa_i)\|f\|_{\Lip}).\end{equation}
By taking the limit $k\to \infty$, we get that 
\[\lim_{k\to \infty}f_k(x)\le \E_{\pi}(f)+\frac{\lambda}{4}D_{\max},\]
and thus 
\[\E_{\pi}(\econst^{\lambda f})=\lim_{k\to \infty} \left(\mtx{P^k} \econst^{\lambda f}\right)(x)\le \econst^{\lambda \E_{\pi}(f)+\lambda^2 D_{\max}/4}.\]
We obtain the bounds \eqref{nonrevgaussbound} and \eqref{nonrevexpbound}  from Lemma \ref{tmaxlambdamaxlemma}.
\end{proof}

\begin{proof}[Proof of Theorem \ref{nonrevconcthm2}]
Similarly to the proof of Theorem 33 of \cite{Ollivier2}, we will bound the moment generating function $\E_{\pi} e^{\lambda f}$ based on the fact that it is the limit of $\mtx{P}^k e^{\lambda f}$, which can be bounded recursively.

Fix some $\lambda>0$, and let $\hat{f}(x):=\hat{f}_0(x):=f(x)/(2M\|f\|_{\Lip})$, and for $k\ge 0$, define $\hat{f}_{k+1}$ as
\[\hat{f}_{k+1}(x):=\mtx{P}\hat{f}_k(x)+\lambda g(\|\hat{f}_k\|_{\Lip}) \cdot S(x).\]
Then Lemma \ref{nonrevlemma} shows that
$(\mtx{P} \econst^{\lambda\hat{f}_k})(x)\le \econst^{\lambda \hat{f}_{k+1}(x)}$, and thus $(\mtx{P}^k \econst^{\lambda\hat{f}})(x)\le \econst^{\lambda \hat{f}_{k}(x)}$.
Now for $k\ge 1$, $\hat{f}_k(x)$ as defined above can be expressed as
\begin{align}\nonumber\hat{f}_k(x)&=\mtx{P}^k(\hat{f})(x)+\lambda\sum_{i=1}^k g(\|\hat{f}_{i-1}\|_{\Lip})\mtx{P}^{k-i}(S)(x),\quad\text{ thus }\\
\label{hatfklipbound}\|\hat{f}_k\|_{\Lip}&\le (1-\kappa_k)\|\hat{f}\|_{\Lip}+\lambda\|S\|_{\Lip}\sum_{i=1}^{k} (1-\kappa_{k-i}) g(\|\hat{f}_{i-1}\|_{\Lip}).
\end{align}
By taking the limit $k\to \infty$, we obtain that
\begin{equation}\label{explambdahatfeq}
\E_{\pi}(\exp(\lambda \hat{f}))\le \lim_{k\to \infty}\exp(\lambda \hat{f}_k(x))\le \exp\left(\lambda \E_{\pi}(\hat{f})+ \lambda^2\left(\sum_{i=0}^{\infty}g(\|\hat{f}_{i}\|_{\Lip})\right)\cdot \E_{\pi}(S)\right).
\end{equation}
In order to proceed, we need to bound $\|\hat{f}_{i}\|_{\Lip}$ and $\sum_{i=0}^{\infty}g(\|\hat{f}_{i}\|_{\Lip})$. 

Using \eqref{kappakleq2}, one can show that $\sum_{j=0}^{\infty}(1-\kappa_j)\le MK/\kappa_K$. 
Define 
\begin{equation}\label{eqlambdamax}
\lambda_{\max}:=\min(1/(3\sigma_{\infty}),\kappa_K/(2KM\|S\|_{\Lip})).
\end{equation}
It is easy to see that $g(x)\le x^2$ for $0\le x\le 1$. This and \eqref{hatfklipbound} implies that for $\lambda\in[0,\lambda_{\max}]$, for any $j\ge 0$, we have $\|\hat{f}_j\|_{\Lip}\le 1$. This is not quite sharp yet, and now we will obtain more precise bounds. Let $F_0:=1$, and for $k\ge 1$, let
\[F_k:=\frac{1}{2}(1-\kappa_K)^{\lfloor k/K\rfloor}+\frac{1}{2}\frac{\kappa_K}{K}\cdot \sum_{i=1}^{k}(1-\kappa_K)^{\lfloor k-i/K\rfloor}\cdot F_{i-1}^2.\]
Then it follows from \eqref{hatfklipbound} that for $\lambda\in [0,\lambda_{\max}]$, $k\ge 0$, we have $\|f_k\|_{\Lip}\le F_k$. Let $G_0:=1, G_1:=(1-\kappa_K/2)$, and for $k\ge 2$, let
\[G_k:=\frac{1}{2}(1-\kappa_K)^{k}+\frac{1}{2}\kappa_K\cdot \left((1-\kappa_K)^{k}+\sum_{j=1}^{k-1}(1-\kappa_K)^{j} G_{k-1-j}^2\right).\]
Then it is easy to see that for $k\ge 0$, we have $F_k\le G_{\lfloor k/K\rfloor+1}$, and by mathematical induction it follows that $G_k\le (1-\kappa_K/2)^{k-1}$ for any $k\ge 1$.
This implies that for $\lambda\in [0,\lambda_{\max}]$,
\begin{equation}\label{hatfkbound}
\|\hat{f}_k\|_{\Lip}\le (1-\kappa_K/2)^{\lfloor k/K\rfloor}, \text{ and }\sum_{i=0}^{\infty}g(\|\hat{f}_{i}\|_{\Lip})\le \frac{K}{\kappa_K-\kappa_K^2/4}.
\end{equation}
Using these two inequalities and \eqref{explambdahatfeq}, we obtain that for $\lambda\in [0,\lambda_{\max}]$,
\begin{align}\label{explambdahatfeq2}
&\E_{\pi}(\exp(\lambda \hat{f}))\le \exp\left(\lambda \E_{\pi}(\hat{f})+ \lambda^2 \cdot \frac{K}{\kappa_K-\kappa_K^2/4}\cdot \E_{\pi}(S)\right),
\end{align}
which implies that for $\lambda\in \left[0,\lambda_{\max}'\right]$, we have
\begin{equation}\label{explambdahatfeq2}
\E_{\pi}(\exp(\lambda f))\le \exp\left(\lambda \E_{\pi}(f) + \lambda^2 \cdot \frac{4M^2\|f\|_{\Lip}^2 K}{\kappa_K-\kappa_K^2/4}\cdot \E_{\pi}(S)\right).
\end{equation}
The tail bounds now follow by Lemma \ref{tmaxlambdamaxlemma}.
\end{proof}

\subsection{Concentration bounds for MCMC empirical averages}\label{secproofmcmc}
We start with the proof of the bias bound.
\begin{proof}[Proof of Proposition \ref{biasboundprop}]
It follows from the definitions that 
\[|\E_{q\mtx{P}^{j}}(f)-\E_{\pi}(f)|\le \|f\|_{\Lip}W_1\left(q\mtx{P}^{j},\pi\right)\le (1-\kappa_j) \|f\|_{\Lip} W_1(q,\pi),\]
summing up from $j=t_0+1$ to $N$, and using $1-\kappa_{t_0+i}\le (1-\kappa_{t_0})(1-\kappa_{i})$ leads to this result.
\end{proof}

We will use the following lemma (a generalisation of Lemma 9 of \cite{Ollivier3}) in the proof of the variance bound (Theorem \ref{MCMCVarRiccithm}).
\begin{lemma}\label{mcmcvarlemma}
For any $\NN\ge 1$, $x\in \Omega$, $f:\Omega\to \R$,
\[\mtx{P}^{\NN}\left(f^2\right)(x)-\left(\mtx{P}^{\NN}(f)(x)\right)^2 \le \|f\|_{\Lip}^2\cdot \sum_{k=0}^{\NN-1}(1-\kappa_{k})^2\cdot \mtx{P}^{\NN-1-k}\left(\frac{\sigma^2}{n}\right)(x).\]
\end{lemma}
\begin{proof}
We omit to note the dependence in $x$ in the following bounds. We have
\begin{align*}
&\mtx{P}^{\NN}(f^2)-\mtx{P}^{\NN}(f)=
\left[\mtx{P}^{\NN-1}\left(\mtx{P}\left(f^2\right)\right)-
\mtx{P}^{\NN-1}\left((\mtx{P}(f))^2\right)\right]+\Big[\mtx{P}^{\NN-2}\left((\mtx{P}(f))^2\right)\\
&-
\mtx{P}^{\NN-2}\left(\left(\mtx{P}^2(f)\right)^2\right)\Big]+\ldots+\left[\mtx{P}\left((\mtx{P}^{\NN-1}(f))^2\right)-\left(\mtx{P}^{\NN}(f)\right)^2\right].
\end{align*}
The result now follows by summing up the inequalities
\[\mtx{P}\left(\left(\mtx{P}^k(f)\right)^2\right)-\left(\mtx{P}^{k+1}(f)\right)^2=\Var_{P_x}(\mtx{P}^k(f))\le (1-\kappa_k)^2\|f\|_{\Lip}^2\cdot \frac{\sigma^2(x)}{n(x)}.\qedhere\]
\end{proof}

\begin{proof}[Proof of Theorem \ref{MCMCVarRiccithm}]
For technical reasons, we are going to cut the sum $Z$ into $K$ parts as $Z=Z_1+\ldots+Z_K$, and then use the bound 
\begin{equation}\label{eqVarqZbreak}\Var_{q}(Z)\le K\sum_{r=1}^{K}\Var_q\left(Z_r\right).\end{equation} For $1\le r\le K$, let $N_r:=\lfloor (N-t_0-r)/K\rfloor+1$, and
\[Z_r:=\frac{\sum_{j=1}^{N_r}f\left(X_{t_0+(j-1)K+r}\right)}{N-t_0}.\]
We are left to bound $\Var_q(Z_r)$ for each $1\le r\le K$. This is similar to the proof of Theorems 2 and 3 of \cite{Ollivier3}. For $x_1,\ldots, x_{N_r}\in \Omega$, let
\[F^{(r)}_{x_1,\ldots,x_{N_r-1}}(x_{N_r}):=\frac{\sum_{j=1}^{N_r}f\left(x_j\right)}{N-t_0},\] 
and define by downward induction for each $1\le i\le N_r-1$,
\[F^{(r)}_{x_1,\ldots,x_{i-1}}(x_i):=\int_{\Omega} F^{(r)}_{x_1,\ldots,x_{i}}(x_{i+1})P_{x_i}^K(\mathrm{d}x_{i+1}).\]
Finally, let
\[F^{(r)}_{\o}(x_1):=\int_{x\in \Omega}F^{(r)}_{x_1}(x_2)P_{x_1}^K(\mathrm{d}x_2).\]
Then by Lemma 3.2 (step 1) of \cite{JoulinPoisson}, we can show that $F^{(r)}_{x_1,\ldots,x_{i-1}}$ is Lipschitz with constant $s_i$ given by 
\begin{equation}\label{Lipschitzconstantsieq}s_i=\frac{\|f\|_{\Lip}}{N-t_0}\sum_{j=0}^{N_r-i}(1-\kappa_K)^{j}\le \frac{\|f\|_{\Lip}}{\kappa_K(N-t_0)}.\end{equation}
Now using Lemma \ref{mcmcvarlemma} in each step (with $\NN=K$), we obtain
\begin{align*}
&\E_{q}\left(Z_r^2\right)=\int_{\Omega^{N_r}}F^{(r)}_{x_1,\ldots,x_{N_r-1}}(x_{N_r})^2 P^{K}_{x_{N_r-1}}(\mathrm{d} x_{N_r})\cdot \ldots \cdot P^{K}_{x_{1}}(\mathrm{d} x_2)\cdot (q\mtx{P}^{t_0+r})(\mathrm{d} x_1)\\
&\quad \le \int_{\Omega^{N_r-1}}F^{(r)}_{x_1,\ldots,x_{N_r-2}}(x_{N_r-1})^2 P^{K}_{x_{N_r-2}}(\mathrm{d} x_{N_r-1})\cdot \ldots \cdot P^{K}_{x_{1}}(\mathrm{d} x_2)\cdot (q\mtx{P}^{t_0+r})(\mathrm{d} x_1)
\\
&\quad +\frac{\|f\|_{\Lip}^2}{\kappa_K^2(N-t_0)^2}\sum_{i=0}^{K-1}(1-\kappa_i)^2\cdot \E_{q\mtx{P}^{t_0+r+(N_r-1)K-i-1}}\left(\frac{\sigma^2}{n}\right)\le \ldots\\
&\quad \le \int_{\Omega}F^{(r)}_{\o}(x_1)^2 (q\mtx{P}^{t_0+r})(\mathrm{d} x_1)\\
&\quad +\frac{\|f\|_{\Lip}^2}{\kappa_K^2(N-t_0)^2}\sum_{i=0}^{K-1}(1-\kappa_i)^2 \sum_{k=1}^{N_r-1}\E_{q\mtx{P}^{t_0+r+kK-i-1}}\left(\frac{\sigma^2}{n}\right).
\end{align*}
Now $F^{(r)}_{\o}(x_1)$ is $\|f\|_{\Lip}/(\kappa_K (N-t_0))$ Lipschitz, so applying Lemma \ref{mcmcvarlemma} with $\NN=t_0+r$ leads to 
\begin{align*}
&\int_{\Omega}F^{(r)}_{\o}(x_1)^2 (q\mtx{P}^{t_0+r})(\mathrm{d} x_1)
\\
&\quad \le 
(\E_q(Z_r))^2+\frac{\|f\|_{\Lip}^2}{\kappa_K^2 (N-t_0)^2}\sum_{k=0}^{t_0+r-1}(1-\kappa_k)^2 \E_{q\mtx{P}^{t_0+r-1-k}}\left(\frac{\sigma^2}{n}\right),
\end{align*}
thus we have
\begin{align*}
\Var_{q}\left(Z_r\right)
&\le \frac{\|f\|_{\Lip}^2}{\kappa_K^2 (N-t_0)^2}\cdot
\Bigg[\sum_{i=0}^{K-1}(1-\kappa_i)^2\sum_{k=1}^{N_r-1}\E_{q\mtx{P}^{t_0+r+kK-i-1}}\left(\frac{\sigma^2}{n}\right) 
\\
&+ \sum_{k=0}^{t_0+r-1}(1-\kappa_k)^2 \E_{q\mtx{P}^{t_0+r-1-k}}\left(\frac{\sigma^2}{n}\right)\Bigg]\\
& \le \frac{\|f\|_{\Lip}^2}{\kappa_K^2 (N-t_0)^2}\cdot \Bigg[\sum_{i=0}^{K-1}(1-\kappa_i)^2\sum_{k=1}^{N_r-1}\E_{q\mtx{P}^{t_0+r+kK-i-1}}(S) \\
& + \sum_{k=0}^{t_0+r-1}(1-\kappa_k)^2 \E_{q\mtx{P}^{t_0+r-1-k}}(S)\Bigg].
\end{align*}
Moreover, it is easy to see that for any $j\in \N$,
\[|\E_{q\mtx{P}^{j}}(S)-\E_{\pi}(S)|\le \|S\|_{\Lip}W_1\left(q\mtx{P}^{j},\pi\right)\le (1-\kappa_j) \|S\|_{\Lip} W_1(q,\pi).\]
Based on inequality \eqref{kappakleq2}, we have 
\begin{align*}
&\sum_{i=0}^{K-1}(1-\kappa_i)^2\sum_{k=1}^{N_r-1}(1-\kappa_{t_0+r+kK-i-1})\le M^3 \cdot 
\frac{K}{\kappa_K}\cdot (1-\kappa_K)^{t_0/K - 1},\\
&\sum_{i=0}^{t_0+r-1}(1-\kappa_i)^2\le \frac{M^2 K}{\kappa_K},\text{ and}\\
&\sum_{k=0}^{t_0+r-1}(1-\kappa_k)^2 (1-\kappa_{t_0+r-1-k})\le M^3 \cdot \frac{K}{\kappa_K}\cdot (1-\kappa_K)^{t_0/K - 3},
\end{align*}
thus
\begin{align*}
&\Var_{q}\left(Z_r\right)\le \frac{\|f\|_{\Lip}^2}{\kappa_K^2 (N-t_0)^2}\cdot 
\left[\sum_{i=0}^{K-1}(1-\kappa_i)^2 (N_r-1)+\II[t_0>0]\cdot \frac{M^2 K}{\kappa_K} \right]\E_{\pi}(S)\\
&+2M^3 \cdot \frac{K}{\kappa_K}\cdot (1-\kappa_K)^{t_0/K - 3}\|S\|_{\Lip} W_1(q,\pi).
\end{align*}
The claim of the theorem now follows by \eqref{eqVarqZbreak}.
\end{proof}

\begin{proof}[Proof of Theorem \ref{MCMCconcRiccithm}]
Let $Z_r$, $N_r$, and $F_{x_1,\ldots,x_{i}}^{(r)}$ be as in the proof of Theorem \ref{MCMCVarRiccithm}. Now $Z=\sum_{r=1}^{K}Z_r$, and by Jensen's inequality,
\begin{equation}\label{MCMCJenseneq}
\E_{q}\left(\econst^{\lambda Z}\right)\le \econst^{\lambda \E_{q}Z}\frac{1}{K}\sum_{r=1}^{K} \E_{q}\left(\econst^{\lambda (Z_r-\E_q(Z_r))}\right).
\end{equation}
This inequality allows us to deduce a bound on the moment generating function of the empirical average $Z$ from bounds on the moment generating functions of $Z_r$. This is a simpler problem, and our approach is similar to the proof of Theorems 4 and 5 of \cite{Ollivier3}.

Let $\hat{C}:=\frac{1}{4M}\cdot \frac{\kappa_K(N-t_0)}{\|f\|_{\Lip}}$, and let $\hat{F}^{(r)}_{x_1,\ldots,x_{i-1}}(x_i):=\hat{C}\cdot F^{(r)}_{x_1,\ldots,x_{i-1}}(x_i)$. Since we have $\|F^{(r)}_{x_1,\ldots,x_{i-1}}\|_{\Lip}\le  \frac{\|f\|_{\Lip}}{\kappa_K(N-t_0)}$ by \eqref{Lipschitzconstantsieq}, it follows that $\|\hat{F}^{(r)}_{x_1,\ldots,x_{i-1}}\|_{\Lip}\le \frac{1}{4M}$.
Let $\hat{Z}_r:=\hat{C}\cdot Z_r$,  $\hat{Z}:=\hat{C}\cdot Z$, and let $R:=\frac{S_{\max}}{(4M)^2}\frac{\econst^{1/6}}{2}\sum_{i=0}^{K-1}(1-\kappa_i)^2$. 

Using Lemma \ref{nonrevlemma} repeatedly $\NN$ times, and the fact that $g(\alpha)\le \frac{\econst^{1/6}}{2} \alpha^2$ for $0\le \alpha\le 1/4$, we obtain that for $\lambda\in \left[0,\frac{1}{3\sigma_{\infty}}\right]$, for any $\alpha$-Lipschitz function $\varphi:\Omega\to \R$ satisfying $\alpha\le 1/(4M)$, we have
\begin{equation}\label{nonrevlemmageneq}(\mtx{P}^{\NN} \econst^{\lambda \varphi})(x)\le \exp\left(\lambda \mtx{P}^{\NN} \varphi(x) + \lambda^2 S_{\max}\cdot \frac{\econst^{1/6}}{2} \alpha^2 \sum_{i=0}^{\NN-1}(1-\kappa_i)^2\right).\end{equation}
In particular, for $\NN=K$, this implies that $(\mtx{P}^K\econst^{\lambda \varphi})(x)\le \exp\left(\lambda \mtx{P}^K \varphi(x) + \lambda^2 R\right)$. Therefore we have
\begin{align*}
&\E_{q}\left(\econst^{\lambda \hat{Z}_r}\right)=\int_{\Omega^{N_r}}\econst^{\lambda \hat{F}^{(r)}_{x_1,\ldots,x_{N_r-1}}(x_{N_r})}P_{x_{N_r-1}}^K(\mathrm{d}x_{N_r})\cdot \ldots \cdot P_{x_1}^K(\mathrm{d}x_{2})\cdot \left(q \mtx{P}^{t_0+r}\right)(\mathrm{d} x_1)\\
&\le \int_{\Omega^{N_r-1}}\econst^{\lambda \hat{F}^{(r)}_{x_1,\ldots,x_{N_r-2}}(x_{N_r-1})+\lambda^2 R}P_{x_{N_r-2}}^K(\mathrm{d}x_{N_r-1})\cdot \ldots \cdot P_{x_1}^K(\mathrm{d}x_{2})\cdot \left(q \mtx{P}^{t_0+r}\right)(\mathrm{d} x_1)\\
& \le \ldots \le \int_{\Omega}\econst^{\lambda \hat{F}^{(r)}_{\o}(x_1)+\lambda^2 (N_r-1)R}(q\mtx{P}^{t_0+r})(\mathrm{d}x_1).
\end{align*}
Here $\hat{F}^{(r)}_{\o}(x_1)$ is $1/(4M)$ Lipschitz, so using \eqref{nonrevlemmageneq} with $\NN=t_0+r$ leads to 
\[\int_{\Omega}\econst^{\lambda F^{(r)}_{\o}(x_1)}(q\mtx{P}^{t_0+r})(\mathrm{d}x_1)\le
\exp(\lambda \E_{q}(Z_r))\cdot \exp\left(\lambda^2 S_{\max}\frac{e^{1/6}}{2(4M)^2}\sum_{i=0}^{t_0+r-1}(1-\kappa_i)^2\right).\]
If $t_0=0$, we have $S_{\max}\frac{e^{1/6}}{2(4M)^2}\sum_{k=0}^{t_0+r-1}(1-\kappa_i)^2\le R$. When $t_0>0$, using \eqref{kappakleq}, we obtain
\[S_{\max}\frac{e^{1/6}}{2(4M)^2}\sum_{i=0}^{t_0+r-1}(1-\kappa_i)^2\le 
R\sum_{k=0}^{\infty}(1-\kappa_K)^{2k}\le \frac{R}{2\kappa_K-\kappa_K^2}\le \frac{R}{\kappa_K}.\]
This implies that for $\lambda\in [0, 1/(3\sigma_{\infty})]$,
\[\E_{q}\left(\econst^{\lambda \hat{Z}_r}\right)\le \exp\left(\lambda\E_{q}(\hat{Z}_r)+\lambda^2 R\left(N_r+\II[t_0>0]/\kappa_K\right)\right),\]
and as previously, using \eqref{MCMCJenseneq}, we have that for $\lambda\in [0, 1/(3K\sigma_{\infty})]$,
\[ 
\E_{q}\left(\econst^{\lambda \hat{Z}}\right)\le \exp\left(\lambda\E_{\pi}(\hat{Z})+\lambda^2 KR\left((N-t_0+K-1)+\II[t_0>0] K/\kappa_K)\right) \right).
\]
This means that for $\lambda\in \left[0, \frac{\kappa_K(N-t_0)}{12M K\sigma_{\infty}\|f\|_{\Lip}}\right]$,
\begin{align*}
&\E_{q}\left(\econst^{\lambda Z}\right)\le \exp\Bigg(\lambda\E_{q}(Z)+\\
&\lambda^2 
\cdot \frac{\econst^{1/6}}{2} \|f\|_{\Lip}^2 S_{\max}\frac{(N-t_0+K-1)+\II[t_0>0] K/\kappa_K}{(N-t_0)^2}\cdot \frac{K}{\kappa_K^2}\sum_{k=0}^{K-1}(1-\kappa_k)^2\Bigg),
\end{align*}
and the bounds follow from Lemma \ref{tmaxlambdamaxlemma}.
\end{proof}

\begin{proof}[Proof of Theorem \ref{MCMCconcRiccithm2}]
The structure of the proof is similar to the proof of Theorem \ref{MCMCconcRiccithm}. We again decompose $Z$ into a sum as $Z=\sum_{r=1}^{K}Z_r$, and use \eqref{MCMCJenseneq} to deduce a bound on the moment generating function of $Z$ from one on the moment generating function of $Z_r$. A key difference is that now we need to use a new version of the inequality \eqref{nonrevlemmageneq}, to be introduced as follows.

Using Lemma \ref{nonrevlemma} repeatedly $\NN\ge 1$ times, and taking into account the inequality \eqref{hatfkbound}, we obtain that for any $\varphi:\Omega\to \R$ satisfying that $\|\varphi\|_{\Lip}\le \frac{1}{2M}$, for any $\lambda\in [0, \min(1/(3\sigma_{\infty}), \kappa_K/(2KM\|S\|_{\Lip})]$, $x\in \Omega$, we have
\begin{equation}\label{nonrevlemmageneq2}\mtx{P}^{\NN}\left(\econst^{\lambda \varphi}\right)(x)\le \exp\left( \lambda \mtx{P}^{\NN}(\varphi)(x)+\lambda^2 \sum_{i=1}^{\NN}\left(1-\frac{\kappa_K}{2}\right)^{2\lfloor \frac{i-1}{K}\rfloor}\cdot \mtx{P}^{\NN-i}(S)(x)\right),\end{equation}
in particular,
$\mtx{P}^{K}\left(\econst^{\lambda \varphi}\right)(x)\le \exp\left( \lambda \mtx{P}^{K}(\varphi)(x)+\lambda^2 \sum_{i=0}^{K-1}\mtx{P}^{i}(S)(x)\right)$.

Let $Z_r$, $N_r$, and $F_{x_1,\ldots,x_{i}}^{(r)}$ be as in the proof of Theorem \ref{MCMCVarRiccithm}. For any $j\ge 1$, \[\left\|\lambda\sum_{k=0}^{jK-1} P^{k}(S)\right\|_{\Lip}\le \lambda MK\|S\|_{\Lip}/\kappa_K,\]
thus for $\lambda\in [0, \min(1/(3\sigma_{\infty}), \kappa_K/(4KM^2\|S\|_{\Lip}))]$, for any $1\le i\le N_r$, we have
\[\left\|\hat{F}^{(r)}_{x_1,\ldots,x_{i-1}}+\lambda\cdot \sum_{k=0}^{(N_r-i)K-1} \mtx{P}^{k}(S)\right\|_{\Lip}\le \frac{1}{2M}.\]
Now by using inequality \eqref{nonrevlemmageneq2} in each step with $\NN=K$, we obtain that
\begin{align*}
&\E_{q}\left(\econst^{\lambda \hat{Z}_r}\right)=\int_{\Omega^{N_r}}\econst^{\lambda \hat{F}^{(r)}_{x_1,\ldots,x_{N_r-1}}(x_{N_r})}P_{x_{N_r-1}}^K(\mathrm{d}x_{N_r})\cdot \ldots \cdot P_{x_1}^K(\mathrm{d}x_{2})\cdot \left(q \mtx{P}^{t_0+r}\right)(\mathrm{d} x_1)\\
&\le \int_{\Omega^{N_r-1}}\econst^{\lambda \hat{F}^{(r)}_{x_1,\ldots,x_{N_r-2}}(x_{N_r-1})+\lambda^2 \sum_{i=0}^{K-1}\mtx{P}^{i}(S)(x)}P_{x_{N_r-2}}^K(\mathrm{d}x_{N_r-1})\cdot \ldots \cdot P_{x_1}^K(\mathrm{d}x_{2})\\
&\cdot \left(q \mtx{P}^{t_0+r}\right)(\mathrm{d} x_1) \le \ldots \le \int_{\Omega}\econst^{\lambda \hat{F}^{(r)}_{\o}(x_1)+\lambda^2 \sum_{k=0}^{(N_r-1)K-1} \mtx{P}^{k}(S)(x_1)}(q\mtx{P}^{t_0+r})(\mathrm{d}x_1)\\
&\le 
\exp\Bigg(\lambda \E_{q}(\hat{Z}_r)\\
&+\lambda^2 \left[\sum_{k=t_0+r}^{(N_r-1) K+t_0+r-1} \E_{q\mtx{P}^{k}}(S)+\sum_{k=1}^{t_0+r}(1-\kappa_K/2)^{2\lfloor (i-1)/K\rfloor} \E_{q\mtx{P}^{t_0+r-i}}(S)\right]\Bigg),
\end{align*}
using inequality \eqref{nonrevlemmageneq2} with $\NN=t_0+r$ in the last step. By \eqref{kappakleq2}, we have \[|\E_{q\mtx{P}^j}(S)-\E_{\pi}(S)|\le \|S\|_{\Lip} W_1(q,\pi) (1-\kappa_j)\le M\|S\|_{\Lip} W_1(q,\pi) (1-\kappa_K)^{\lfloor j/K\rfloor},\]
which implies that
\begin{align*}
&\sum_{k=t_0+r}^{(N_r-1) K+t_0+r-1} \E_{q\mtx{P}^{k}}(S)+\sum_{k=1}^{t_0+r}(1-\kappa_K/2)^{2\lfloor (i-1)/K\rfloor} \E_{q\mtx{P}^{t_0+r-i}}(S) \\
&\le \E_{\pi}(S)\cdot \left((N-t_0)+\II[t_0>0]\cdot\frac{2K}{\kappa_K}\right)\\
&+M\|S\|_{\Lip} W_1(q,\pi) \cdot \frac{3K}{\kappa_K}\cdot\left(1-\kappa_K/2\right)^{\lfloor t_0/K\rfloor-3},
\end{align*}
and thus for $\lambda\in [0, \min(1/(3\sigma_{\infty}), \kappa_K/(4KM^2\|S\|_{\Lip}))]$,
\begin{align*}&\E_{q}\left(\econst^{\lambda \hat{Z}_r}\right) \le \exp\left(\lambda \E_{q}(\hat{Z}_r)\right)\cdot \exp\Bigg(\lambda^2 \Bigg[\E_{\pi}(S)\bigg((N-t_0)
+\II[t_0>0]\cdot\frac{2K}{\kappa_K}\bigg)
\\
&+M\|S\|_{\Lip} W_1(q,\pi) \cdot \frac{3K}{\kappa_K}\cdot\left(1-\kappa_K/2\right)^{\lfloor t_0/K\rfloor-3}\Bigg]\Bigg).
\end{align*}
Now using \eqref{MCMCJenseneq}, and using $\hat{Z}=\hat{C}Z$, we obtain that for 
\[\lambda\in \left[0, \frac{1}{4MK}\cdot \frac{\kappa_K(N-t_0)}{\|f\|_{\Lip}} \cdot\min(1/(3\sigma_{\infty}), \kappa_K/(4KM^2\|S\|_{\Lip}))\right],\]
we have
\begin{align*}&\E_{q}\left(\econst^{\lambda Z}\right) \le \exp\left(\lambda \E_{q}(Z)\right)\\
&\cdot \exp\Bigg(\lambda^2\cdot \frac{16\|f\|_{\Lip}^2M^2 K^2}{\kappa_K^2(N-t_0)^2} \cdot \Bigg[\E_{\pi}(S)
\cdot \left((N-t_0)+\II[t_0>0]\cdot\frac{2K}{\kappa_K}\right)\\&+M\|S\|_{\Lip} W_1(q,\pi) \cdot \frac{3K}{\kappa_K}\cdot\left(1-\kappa_K/2\right)^{\lfloor t_0/K\rfloor-3}\Bigg]\Bigg).
\end{align*}
The tail bounds now follow by Lemma \ref{tmaxlambdamaxlemma}.
\end{proof}

\section*{Acknowledgements}
We thank the referee for his insightful comments and careful reading of the manuscript.
We thank Malwina Luczak, Yann Ollivier, and Laurent Veysseire for their comments. Finally, many thanks to Roland Paulin for the enlightening discussions.

\bibliographystyle{imsart-nameyear}
\bibliography{References}

\end{document}